\documentclass[onefignum,onetabnum]{siamart190516}

\usepackage{pgfplots}
\pgfplotsset{compat=newest}

\usepackage{lipsum}
\usepackage{amsfonts}
\usepackage{graphicx}
\usepackage{epstopdf}
\usepackage{algorithmic}
\ifpdf
  \DeclareGraphicsExtensions{.eps,.pdf,.png,.jpg}
\else
  \DeclareGraphicsExtensions{.eps}
\fi

\newsiamremark{remark}{Remark}
\newsiamremark{problem}{Problem}
\newsiamremark{hypothesis}{Hypothesis}
\newsiamremark{assumption}{Assumption}
\crefname{hypothesis}{Hypothesis}{Hypotheses}
\newsiamthm{claim}{Claim}

\headers{Classical System Theory and Turnpike for Descriptor Systems}{J. Heiland and E. Zuazua}

\title{Classical System Theory Revisited for Turnpike in Standard State Space
  Systems and Impulse Controllable Descriptor Systems%
}

\author{Jan Heiland\thanks{Max Planck Institute for Dynamics of Complex
    Technical Systems, Magdeburg, Germany 
  (\email{heiland@mpi-magdeburg.mpg.de}, \url{https://www.janheiland.de}).}
\and Enrique Zuazua\thanks{Chair in Applied Analysis, Alexander von
  Humboldt-Professorship, Department of Mathematics,
  Friedrich-Alexander-Universit\"at Erlangen-N\"urnberg, 91058 Erlangen,
  Germany (\email{enrique.zuazua@fau.de}),\newline \indent \hspace{4pt} Chair of Computational Mathematics, Fundaci\'on Deusto, University of Deusto,
48007 Bilbao, Basque Country, Spain,
\newline \indent \hspace{8pt}Departamento de Matem\'aticas, Universidad Aut\'onoma de Madrid, 28049 Madrid, Spain.}
}

\usepackage{amsopn}

\providecommand{\inva}[1]{\text{~\textup{d}} #1~}
\providecommand{\Rmat}[2]{\mathbb R^{#1\times #2}}
\providecommand{\Rvec}[1]{\mathbb R^{#1}}
\DeclareMathOperator{\rank}{rank}
\DeclareMathOperator{\const}{const}

\def\xfh{\ensuremath{x_h}}

\def\xs{\ensuremath{x_{s}}}
\def\ws{\ensuremath{w_{s}}}
\def\us{\ensuremath{u_{s}}}
\def\PP{\ensuremath{P_+}}

\def\PD{\ensuremath{P_{\Delta}}}
\def\AP{\ensuremath{A_+}}
\def\APmo{\ensuremath{\AP^{-1}}}
\def\APs{\ensuremath{\AP^{*}}}
\def\APms{\ensuremath{\AP^{-*}}}

\def\ye{y_e}
\def\yc{y_c}
\providecommand\mxp[1]{e^{\{#1\}}}

\def\daE{\ensuremath{\mathcal E}}
\def\daA{\ensuremath{\mathcal A}}
\def\daF{\ensuremath{\mathcal F}}
\def\daC{\ensuremath{\mathcal C}}
\def\daB{\ensuremath{\mathcal B}}
\def\daP{\ensuremath{\mathcal P}}
\def\daPP{\ensuremath{\mathcal P_+}}
\def\daPPs{\ensuremath{\mathcal P_+^*}}
\def\daPs{\ensuremath{\mathcal P^*}}
\def\daPD{\ensuremath{\mathcal P_\Delta}}
\def\daPDs{\ensuremath{\mathcal P_\Delta^*}}
\def\daAP{\ensuremath{\mathcal A_+}}
\def\daAPs{\ensuremath{\mathcal A_+^*}}
\def\daSi{\ensuremath{S_{1}}}

\def\bo{B_1}
\def\bo{B_1}
\def\co{C_1}
\def\ct{C_2}
\def\bt{B_2}
\def\bt{B_2}
\def\ao{A_{11}}
\def\ato{A_{21}}
\def\aot{A_{12}}
\def\at{A_{22}}

\def\apo{A_{+;1}}
\def\apto{A_{+;21}}
\def\apot{A_{+;12}}
\def\apt{A_{+;2}}
\def\aptms{A_{+;2}^{-*}}
\def\aptmo{A_{+;2}^{-1}}

\def\boo{B_1B_1^*}
\def\bto{B_2B_1^*}
\def\btos{B_1B_2^*}
\def\btt{B_2B_2^*}

\def\daPDii{\ensuremath{P_{\Delta;1}}}
\def\daPDit{\ensuremath{P_{\Delta;21}}}
\def\daPPii{\ensuremath{P_{+;1}}}
\def\daPPtt{\ensuremath{P_{+;2}}}
\def\ddaPDii{\ensuremath{\dot P_{\Delta;1}}}
\def\daPo{\ensuremath{P_{1}}}
\def\daPt{\ensuremath{P_{2}}}
\def\daPto{\ensuremath{P_{21}}}

\def\bA{\bar A}
\def\bAs{\bar A^*}
\def\bAms{\bar A^{-*}}
\def\bB{\bar B}
\def\bC{\bar C}
\def\bW{\bar W}
\def\bsigma{\bar \lambda}

\def\xsss{x_{\textup{sss}}}
\def\xfss{x_{\textup{fss}}}
\def\bsss{B_{\textup{sss}}}
\def\bfss{B_{\textup{fss}}}

\def\tbs{\bar{\tilde{S}}}

\begin{document}

\maketitle

\begin{abstract}
  The concept of turnpike connects the solution of long but finite time horizon
  optimal control problems with steady state optimal controls. A key ingredient
  of the analysis of turnpike phenomena is the linear quadratic regulator problem and
  the convergence of the solution of the associated differential Riccati
  equation as the terminal time approaches infinity. This convergence has been
  investigated in linear systems theory in the 1980s. We extend classical system
  theoretic results for the investigation of turnpike properties of standard
  state space systems and descriptor systems. We present conditions for turnpike
  phenomena in the nondetectable case and for impulse controllable descriptor systems. For
  the latter, in line with the theory for standard linear systems, we establish
  existence and convergence of solutions to a generalized differential Riccati
  equation.
\end{abstract}



\section{Introduction}

The notion of \emph{turnpike} has been used in economics since long and in
control theory (see, e.g., the textbooks \cite{Zas06, Zas15}) for about 40 years with an increasing interest in the last decade. 
Turnpike denotes the property of the control and the solution to a finite time
optimization problem to be close to the optimal values for the associated steady
state problem most of the time.

The backbone of most turnpike results for time autonomous systems is the
turnpike property of a relevant linear quadratic regulator (LQR) optimization problem. And the
turnpike property of the LQR problem is intimately linked to the decay of the solution to
the associated generalized differential Riccati equation towards the stabilizing
solution of an algebraic Riccati equation; see e.g. \cite[Lem. 2.6]{PorZ13}.

In the first part of this manuscript, we use classical mathematical systems
theoretic results as presented by Callier, Willems, and Winkin \cite{CalWW94} to show the
turnpike property of the LQR problem. 
For that we extend the results to the affine linear optimal control problem
using an explicit formula of the state transition matrices of the closed loop
system. 
Having connected the system theoretic toolbox to the investigation of turnpike
behaviors, we can immediately provide new general results for cases where the
system is not detectable which is a current research issue; see \cite{GrG21, PigS20}.

In the second part of the paper, we derive turnpike properties of linear
quadratic optimal control problems that are constrained by descriptor systems.
Descriptor systems are also commonly referred to as differential algebraic
equations (DAEs). To our best knowledge turnpike phenomena of DAEs have not been addressed
so far.

In analogy with the standard LQR case, we will link the asymptotic behavior of an
associated differential Riccati equation to the turnpike property.

Early considerations on Riccati equations for DAEs were made in \cite{KokY72} by
examining LQ-regulators for singularly perturbed ordinary differential
equations. An extensive
investigation and fundamental results for the finite time LQR problem for DAEs
has been provided by Bender and Laub \cite{BenL87} who expanded also on the work
of Cobb \cite{Cob83} and Pandolfi \cite{Pan81}. 

Bender and Laub defined several equivalent relevant Riccati equations in
standard state space form; see \cite[Sec. IV]{BenL87}. A generalized Riccati
equation which, in particular, can be stated in the original system
coefficients is not addressed in \cite{BenL87} apart from noting
that the most obvious symmetric formulation is not well suited. 
We also mention the extension of the work by \cite{BenL87} to endpoint
constraints \cite{ZhaJZ90}. 
The nonsymmetric differential Riccati equation that is formulated in the original
coordinate system and that also is a main subject of this paper has been treated in
\cite{KatM92}.
There the relation the LQR problem for descriptor system has been discussed and
the existence of solutions under general conditions has been shown.
Before, this nonsymmetric differential generalized Riccati
equation has been considered in \cite{Kur84}, where necessary conditions for
existence of solutions in general and sufficient conditions for some special
cases were derived. 

The related nonsymmetric generalized algebraic Riccati has been investigated in
\cite{KawTK99} and applied in the context of model reduction for infinite
time-horizon control systems in
\cite{MRS11}. 

Apart from this branch, the literature on the DAE LQR optimization
problem on finite time horizons has been enriched with results on suitable
reformulations of the optimality conditions \cite{KunM97}, on particularly
structured cases \cite{Bac03,Hei16}, and on the problem with time-varying
coefficients \cite{KunM08,KunM11,KurM07}. 
In the course of the investigations, several formulations of generalized
differential Riccati equations have been proposed; see the discussion in
\cite{KurM07}. 

More recently based on a generalized \emph{Kalman-Yakobovich-Popov} inequality
\cite{ReiRV15} and the \emph{Lur'e} equations, which provide a true
generalization (cp. \cite[Sec. 6]{ReiV19}) to the presented Riccati based
approach for DAEs, the linear quadratic regulation problem for DAEs has been
considered in great generality; see \cite{ReiV19}. 

This work contributes to the theory on Riccati equations for descriptor systems in the
following respects. 
We show that under the conditions used in \cite{BenL87, KatM92} and an
additional definiteness condition on the optimization problem, the
solution of the nonsymmetric generalized differential Riccati equations 
has a distinguished structure and converges to the stabilizing solution of the
associated algebraic Riccati equation. This structure also implies that the
provided optimal feedback gains make the closed-loop system
\emph{impulse-free} so that they are a \emph{best choice} according to a
conjecture stated in \cite[Sec. VII]{BenL87}. With the convergence of the gains and the closed-loop being impulse-free, we
then can show that the DAE constrained LQ optimization problem has the turnpike
property. 

The line of arguments and results in this paper are as follows. In Section
\ref{sec:basic-notions-results}, we introduce the linear quadratic regulator
(LQR) problem for standard state space systems, the notion of turnpike, and
classical results on the asymptotic behavior of the solutions and the controls
that immediately imply well known turnpike results. 
Next, in Section \ref{sec:ode-lqr-affine}, we derive explicit formulas for the
solutions to the affine LQR problem, i.e. the LQR problem with nonzero target
states. Then the arguments of the first section can be applied to conclude turnpike
properties also in this case. In the second part of the paper, we consider the
LQR problem with DAE constraints. Therefore, we introduce the relevant concepts
in Section \ref{sec:lqr-dae} and prove existence and asymptotic decay of
solutions of the generalized differential Riccati equation in Section \ref{sec:exist-decay-gDRE}. Finally,
we can prove turnpike properties of the affine LQR problem with DAE constraints
in Section \ref{sec:dae-affine-lqr}. We conclude the paper with summarising
remarks and an overview of related open research questions.

\section{Basic Notations, Notions and Results for the Linear Quadratic Regulator
Problem}\label{sec:basic-notions-results}

% Names:
% 
% \begin{itemize}
%   \item $x$~... solution to the finite time problem
%   \item \xfh~... solution to the finite time homogeneous problem
%   \item \xs~... steady state solution
%   \item \us~... steady state optimal control
%   \item \ws~...  affine part of steady state optimal control
%   \item \xih~... solution to the infinite time homogeneous problem
%   \item $S=F^*F$ ... terminal condition for the DRE
%   \item $P$ ... solution to the DRE
%   \item \PP~... solution to the ARE
%   \item $W$ ... solution to the closed loop Lyapunov equation
%   \item \AP~... closed loop matrix $\AP:=A-BB^*\PP$
% \end{itemize}

We consider the finite time horizon linear quadratic optimization problem.
\begin{problem}[Finite horizon optimal control problem]\label{prb-fiti-optcont}
  For coefficients $A\in \Rmat nn$, $B\in\Rmat nm$, $C\in\Rmat kn$, and $F\in
  \Rmat \ell n$$,$ for an initial value $x_0\in \Rvec n$, for target outputs
  $\yc \in \Rvec k$ and $\ye \in \Rvec \ell$ and a terminal time $t_1>0$,
  consider the optimization of the cost functional
\begin{equation*}
  \frac 12 \int_{0}^{t_1} \|Cx(s)-\yc \|^2+ \|u(s)\|^2\inva s +\frac 12  \|F x(t_1)-\ye \|^2
  \to \min_u
\end{equation*}
subject to
\begin{equation*}
  \dot x(t) = Ax(t) + Bu(t), \quad x(0)=x_0.
\end{equation*}
\end{problem}
We will investigate how solutions to Problem \ref{prb-fiti-optcont} will relate
to solutions of the related state optimization problem in particular for large time
horizons $t_1$.
\begin{problem}[Steady state optimal control problem]\label{prb-stst-optcont}
Consider
\begin{equation*}
  \frac 12 \|Cx-\yc\|^2+ \frac 12\|u\|^2\inva s
  \to \min_u
\end{equation*}
subject to
\begin{equation*}
  0 = Ax + Bu. 
\end{equation*}
\end{problem}

% \todo{\texttt{\#wontfix}: Existence of solutions to the optimization problem. }

If $\yc=0$ and $\ye=0$, then we will refer to Problems \ref{prb-fiti-optcont}
and \ref{prb-stst-optcont} as \emph{homogeneous} LQR problems, otherwise as
\emph{affine} LQR problems.

\begin{definition}\label{def:turnpike}
  The finite time optimal control problem has the \emph{(exponential) turnpike} property, if
  for some constant vectors $x_s$ and $u_s$ it holds that
  \begin{equation*}
    \|x(t) - \xs\| \leq \const (\mxp{\lambda t} + \mxp{\lambda (t_1 -t)})
  \end{equation*}
  and
  \begin{equation*}
    \|u(t) - \us\| \leq \const (\mxp{\lambda t} + \mxp{\lambda (t_1 -t)})
  \end{equation*}
  for $t\leq t_1$ and for constants $\const>0$ and $\lambda<0 $ independent of $t_1$.
\end{definition}

\begin{remark}
  Throughout this manuscript, the notation $\const$ will be used to denote a generic constant
  value that is independent of $t$ and $t_1$ but unspecified otherwise.
\end{remark}

In general terms, the \emph{turnpike property} means the existence of a
trajectory -- the \emph{turnpike} -- that is defined on the whole timeline and
only depends on the optimization criterion such that the solution to the
optimal control problem on finite time interval is close to 
the turnpike except at the beginning or the end of the interval; cp. the preface
in \cite{Zas06}. In general considerations, the \emph{closedness} is
characterized through the measures of the time intervals in which the solution
leaves an $\epsilon$-neighborhood of the turnpike (cp. \cite[p. xviii]{Zas06})
which can be time-dependent; see, e.g., \cite{FauFOW20,TreZZ18}.

In this work, we consider the exponential turnpike property (as in Definition
\ref{def:turnpike}) with a steady state
as the turnpike, which is the solution to an associated steady state
optimal control problem; see also, e.g., \cite{TreZ15}. When considering differential algebraic equations
below, however, we will lay out that the notion of \emph{an} associated steady state
problem might be not uniquely defined.

% In Definition \ref{def:turnpike} the vague requirement that $C$
% should be of moderate size means that an exponential approach should be visible
% also on a possibly moderate time horizon.
We start with the fundamental assumption for our considerations.

\begin{assumption}\label{ass:exist-are-sol-stab}
  We assume that $A$, $B$, and $C$ in Problems \ref{prb-fiti-optcont} and
  \ref{prb-stst-optcont} are such that the \emph{algebraic Riccati equation} 
  \begin{equation}\label{eqn:alg-ric}
    A^*X + XA - XBB^*X + C^*C = 0
  \end{equation}
  has a stabilizing solution $\PP\in\Rmat nn$ which means that the eigenvalues of 
  \begin{equation}\label{eqn:app}
    \AP := A - BB^*\PP 
  \end{equation}
  all have negative real part.
\end{assumption}

We note that \AP~is invertible and that its \emph{spectral abscissa} $\lambda$,
i.e. the maximum real part among the eigenvalues of $\AP$, is strictly less than
zero.

\begin{remark}\label{rem:are-detectability}
  Assumption \ref{ass:exist-are-sol-stab} requires that $(A, B)$ are
  \emph{stabilizable}. Detectability of $(C,A)$ is not required, since \PP~may
  exist in the case where detectability is not given, see, e.g.,
  \cite{Mar71,Kuc73}.
\end{remark}

\begin{lemma}\label{lem:steady-state-sols}
  Under Assumption \ref{ass:exist-are-sol-stab}, the solution to the steady
  state optimal control problem is given as
\begin{align*}
  \xs &= (A-BB^*\PP)^{-1}BB^*(A^*-\PP BB^*)^{-1}C^*\yc \\
      &= \APmo BB^* \APms C^*\yc
  \intertext{and}
  u_s &= -B^*\PP \xs - B^*\ws ,
  \intertext{where}
    w_s &= \APms C^*\yc.
\end{align*}
\end{lemma}

\begin{proof} 
  For Problem \ref{prb-stst-optcont} as a linear-quadratic constrained
  optimization problem, it holds that $(\xs, \us)$ is an optimal solution, if,
  and only if, there exists a \emph{Lagrange-multiplier} $\lambda_s$ such that
  \begin{subequations}\label{eqn:fon}
    \begin{align}
      0 &= A\xs + B\us, \label{eqn:fon-varl} \\
      0 &= A^*\lambda_s + C^*(C\xs-\yc), \label{eqn:fon-varx}\\
      0 &= B^* \lambda_s + \us. \label{eqn:fon-varu}
    \end{align}
  \end{subequations}
  One can confirm directly that with $\lambda_s:= \PP \xs + \ws$, with $\AP$
  invertible, and with the vectors $\xs$ and $\us$, relations
  \eqref{eqn:fon-varl} and \eqref{eqn:fon-varu} are fulfilled. To see that
  \eqref{eqn:fon-varx} is fulfilled too, we compute
  \begin{align}
    A^*\lambda_s + C^*(C\xs-\yc)   \notag
    &= A^*\PP \xs - A^*\ws + C^*(C\xs-\yc)\notag \\
&= A^*\PP \APmo BB^* \APms C^*\yc - A^*\APms C^*\yc +\notag \\
&\phantom{= A^*\PP \APmo BB^* \APms C^*\yc}+C^*C\APmo BB^* \APms C^*\yc-C^*\yc
\notag \\
&= (A^*\PP \APmo BB^* \APms  - A^*\APms  + C^*C\APmo BB^* \APms - I)C^*\yc
\notag  \\
&= ((A^*\PP   + C^*C)\APmo BB^* \APms - A^*\APms - I)C^*\yc \notag  \\
&= ((-\PP A   +\PP BB^*\PP )\APmo BB^* \APms - A^*\APms - I)C^*\yc
\label{eqn:here-we-use-ARE} \\
&= (-\PP (A   -\PP BB^*\PP )\APmo BB^* \APms - A^*\APms - I)C^*\yc  \notag \\
&= (-\PP BB^* \APms - A^*\APms - I)C^*\yc  \notag \\
&= (-(\PP BB^* - A^*)\APms - I)C^*\yc  \notag \\
&= (-(-I) - I)C^*\yc  =0, \notag
  \end{align}
where in \eqref{eqn:here-we-use-ARE} we have used, that $\PP$ fulfills the ARE
\eqref{eqn:alg-ric}.
\end{proof}

The turnpike property is intimately linked to the convergence of the solution
$P$ to the differential Riccati equation towards the stabilizing solution \PP~of
the associated algebraic Riccati equation. In fact, this convergence appears as
a necessary condition for the turnpike property in linear quadratic systems in the
fundamental work by Porreta and Zuazua \cite[Cor. 2.7]{PorZ13}. On the other
hand, basic system theoretic investigations of the convergence of $P$
towards \PP, as presented in \cite{CalWW94}, resulted in the formula
\begin{equation}\label{eqn:calww63}
  \| \xfh(t) -  \mxp{t\AP}x_0 \|  \leq \const \mxp{\lambda t_1}\mxp{\lambda (t_1-t)};
\end{equation}
with $\lambda<0$ being the \emph{spectral abscissa} of $\AP$~and where $x_h$ is
the solution to Problem \ref{prb-fiti-optcont} with $\yc=0$ and $\ye=0$; cp.
\cite[Thm. 4]{CalWW94}. 
Since $\mxp{t\AP}x_0$ goes to zero exponentially with rate $\lambda$ and since
$x_s=0$ for the homogeneous problem with $y_c=0$, by means of \eqref{eqn:calww63}, one
can directly infer the turnpike property for the homogeneous case:
\begin{equation}\label{eqn:tp-from-calww}
  \begin{split}
    \| \xfh(t) -  0 \|  &\le \| \xfh(t) -  \mxp{t\AP}x_0 \|  + \| \mxp{t\AP}x_0 \| \\
                        &\leq \const  \mxp{\lambda t_1}\mxp{\lambda (t_1-t)} +
                        \const \mxp{\lambda t} \le \const (\mxp{\lambda (t_1-t)} + \mxp{\lambda t}).
\end{split}
\end{equation}

In view of extending this result from \cite{CalWW94} to the affine case,
we recall well-known links between the solution to the
finite time optimal control problem and the differential Riccati equation
combined with a feedforward term.  % that accounts for the inhomogeneities.

\begin{theorem}[Ch. 3.1 of \cite{Loc01}]\label{thm:locatelli}
  The solution to Problem \ref{prb-fiti-optcont} is given as $(x, u)$ where 
  \begin{equation*}
    u(t) = -B^*(P(t)x(t)+w(t)),
  \end{equation*}
  where $P$ is the unique solution to the differential Riccati equation (DRE)
\begin{equation}\label{eqn:DRE}
  -\dot P(t) = A^*P(t) + P(t)A - P(t)BB^*P(t) + C^*C, \quad P(t_1)=F^*F,
\end{equation}
where $w$ is the solution to
\begin{equation}\label{eqn:w-feedforward}
  -\dot w(t) = (A^* - P(t)BB^*)w(t) - C^*\yc, \quad w(t_1) = -F^*\ye,
\end{equation}
and where $x$ solves
\begin{equation}\label{eqn:optix-closedloop}
  \dot x(t) = (A-BB^*P(t))x(t) - BB^*w(t), \quad x(0)=x_0.
\end{equation}
\end{theorem}

In what follows we will use the abbreviation $S=F^*F$.

Note that in Theorem \ref{thm:locatelli} that characterizes the optimal controls
for finite times, stability does not play a role so that the coefficients
$(A,B,C)$ and $F$ can be arbitrary. In order to link to the steady state,
however, we will require Assumption \ref{ass:exist-are-sol-stab} to hold. In
this case, namely if \PP~exists, the following quantities are well-defined; see
(see \cite[Lem. 1, Lem. 5]{CalWW94}):

\begin{enumerate}
  \item The \emph{closed loop reachability Gramian}
  \begin{equation}\label{eqn:def-reach-gram}
    W:= \int_0^\infty \mxp{s\AP}BB^*\mxp{s\APs} \inva s,
  \end{equation}
\item the \emph{closed loop reachability Gramian} on $[0,\tau]$
  \begin{equation}\label{eqn:def-fintim-reach-gram}
  W(\tau) = \int_0^\tau \mxp{s\AP}BB^*\mxp{s\APs} \inva s = W-\mxp{\tau\AP}W\mxp{\tau\APs}
\end{equation}
\item as well as the \emph{sliding terminal condition}.
  \begin{equation}\label{eqn:def-tilde-s}
  \tilde S(\tau) := (S-\PP)[I+W(\tau)(S-\PP)]^{-1}
\end{equation}
\end{enumerate}

For the latter, the following Lemma is relevant:

\begin{lemma}[\cite{CalWW94}, Lem. 5]\label{lem:stilde}
  Let Assumption \ref{ass:exist-are-sol-stab} hold and consider $W$, $W(\tau)$,
  and $\tilde S$ as defined in \eqref{eqn:def-reach-gram},
  \eqref{eqn:def-fintim-reach-gram}, and \eqref{eqn:def-tilde-s}. If
  $[I+W(S-\PP)]$ is invertible, then $\tau \to \tilde S(\tau)$ is a decreasing
  function and for any $\tau \ge 0$, meaning that
  \begin{equation*}
    S-\PP = \tilde S(0) \ge \tilde S(\tau) \ge \tilde S(\infty) =
    [I+W(S-\PP)]^{-1}.
  \end{equation*}
  Moreover, for the spectral norm it holds that
  \begin{equation*}
    K(\tilde S) := \sup_{\tau \ge 0} \|\tilde S(\tau) \| = \max \{\|S-\PP\|, \| \tilde
    S(\infty)\| \}.
  \end{equation*}
\end{lemma}

\begin{remark}\label{rem:uniformW}
  For the spectral norm of the Gramians it holds that
  \begin{equation*}
    \|W(\tau)\| \leq \|W\|.
  \end{equation*}
\end{remark}

The condition that $[I+W(S-\PP)]$ is invertible, was shown to be necessary and
sufficient for the convergence of $P(t) \to \PP$ as $t_1 \to \infty$; see
\cite[Thm. 2]{CalWW94}. For what follows we will assume that this condition holds.

\begin{assumption}\label{ass:szero-invertible}
  If Assumption \ref{ass:exist-are-sol-stab} holds and $W$, $W(\tau)$,
  and $\tilde S$ are as defined in \eqref{eqn:def-reach-gram}, then, with $S:=F^*F$, the matrix
  \begin{equation*}
    [I+W(S-\PP)]
  \end{equation*}
  is invertible, where $F$ defines the terminal constraint in
  Problem \ref{prb-fiti-optcont}.
\end{assumption}

\begin{remark}\label{rem:fandc}
Most literature on turnpike properties of optimal control problems assume that
$(C,A)$ is detectable, which is a sufficient condition for Assumption
\ref{ass:szero-invertible}. However, the undetectable subspace for $(C,A)$ can
be compensated for if the nullspace of the terminal cost $F$ only has the
trivial intersection with it, which provides a necessary and sufficient condition
for the convergence of $P(t)$ towards $\PP$; see \cite[Thm. 2]{CalWW94}.
\end{remark}

\begin{remark}
  In \cite[Thm. 8.4]{GrG21}, the turnpike property of an LQ optimization problem
  with an undetectable pair $(C,A)$ has been established by imposing state
  constraints and the absence of unobservable modes on the imaginary axis.
  These assumptions exclude oscillatory modes or unstable modes that are not detected
  by $C$. In our case, unobservable unstable modes are \emph{fixed} by the
  terminal constraint; cp. Assumption \ref{ass:szero-invertible}.
  With the assumption of stabilizability, which, in particular, excludes uncontrollable modes on
  the imaginary axis, the assumptions of \cite[Thm. 8.4]{GrG21} on $(A,B,C)$ are
  equivalent to assuming the existence of a stabilizing solution to the ARE as
  in Assumption \ref{ass:exist-are-sol-stab}.
\end{remark}

We illustrate the implications of Remark \ref{rem:fandc} in a numerical example.
Consider Problem \ref{prb-fiti-optcont} with $t_1=10$, $\yc=0$, and $\ye=1$ and
with the coefficients
\begin{equation}\label{eqn:example-abc}
  A = 
  \begin{bmatrix}
    2 & 0 \\ 0 & -1 
  \end{bmatrix},\quad
  B = 
  \begin{bmatrix}
    1 \\ 1 
  \end{bmatrix},\quad
  C = 
  \begin{bmatrix}
    0 & \sqrt{3}
  \end{bmatrix},
\end{equation}
borrowed from an example in \cite[pp. 31]{Mar71}.
Here, $(A,B)$ is controllable and, thus, stabilizable, while $(C,A)$ is not
detectable. Still, a stabilizing solution to the associated
algebraic Riccati equation \eqref{eqn:alg-ric} exists. 

Then, as implied by the
conditions laid out in Remark \ref{rem:fandc}, the solution to the differential
Riccati equation that starts in $F^*F$ converges to the stabilizing solution if,
and only if, the nullspace of $F$ and the space that is not detected by $(C, A)$ --
which in this case is spanned by $\begin{bmatrix} 1 & 0 \end{bmatrix}^*$ --
intersect only trivially.

Accordingly, with the choice $F=C$, the solution of the differential Riccati
equation converges to 
%$ \begin{bmatrix} 0 & 0 \\ 0 &1 \end{bmatrix}$ which is 
a symmetric positive
definite solution to the ARE that, however, is not stabilizing. Also, the
associated optimal state $x$ does not satisfy the turnpike property as it can be
seen from the logarithmic plot of $|x|$ in the first row of Figure
\ref{fig:turnpikenno}.

Vice versa, with the choice of $F=\begin{bmatrix} \sqrt{3} & 0 \end{bmatrix}$,
Assumption \ref{ass:szero-invertible} holds, the solution to the DRE converges
to a stabilizing solution of the ARE, and the optimal state $x$ satisfies the
turnpike estimate; cp. the second row of Figure \ref{fig:turnpikenno}.

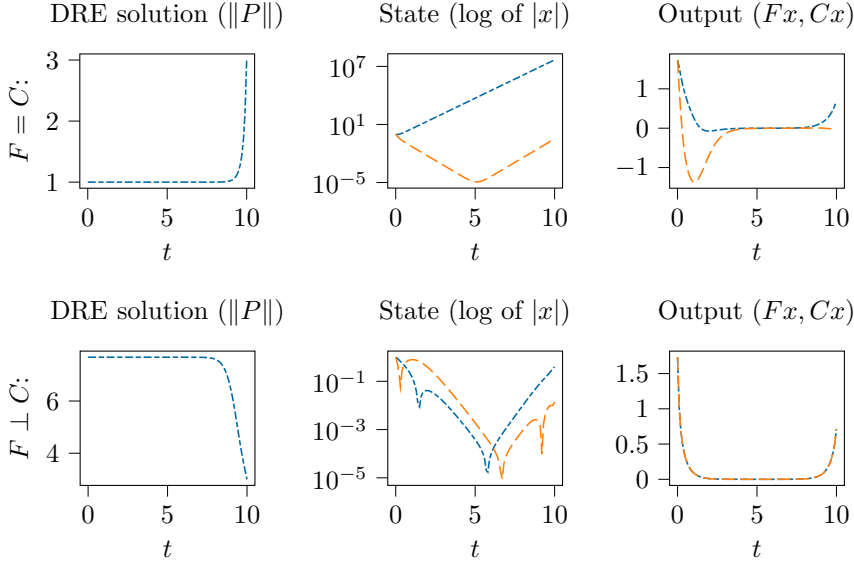
\begin{figure}
  \begin{minipage}[t]{0.3\linewidth}
      % The case of $F=C$:\\[.2in]
      % This file was created by tikzplotlib v0.9.4.
\begin{tikzpicture}

\definecolor{color0}{rgb}{0,0.419607843137255,0.643137254901961}

\begin{axis}[
tick align=outside,
tick pos=left,
title={DRE solution \(\displaystyle (\|P\|\))},
width=\textwidth,
x grid style={white!69.0196078431373!black},
xlabel={\(\displaystyle t\)},
xmin=-0.5, xmax=10.5,
xtick style={color=black},
y grid style={white!69.0196078431373!black},
ylabel={\(\displaystyle F=C\):},
ymin=0.899999999010623, ymax=3.10000000004711,
ytick={1, 2, 3},
ytick style={color=black}
]
\addplot [semithick, color0, dash pattern=on 3pt off 1pt on 2pt off 1pt]
table {%
0 0.999999999924208
0.1 0.99999999992279
0.2 0.999999999921373
0.3 0.999999999919955
0.4 0.999999999918538
0.5 0.99999999991712
0.6 0.999999999915703
0.7 0.999999999914285
0.8 0.999999999912868
0.9 0.99999999991145
1 0.999999999910033
1.1 0.999999999908615
1.2 0.999999999907197
1.3 0.99999999990578
1.4 0.999999999904363
1.5 0.999999999902945
1.6 0.999999999901527
1.7 0.99999999990011
1.8 0.999999999898692
1.9 0.999999999897275
2 0.999999999895857
2.1 0.99999999989444
2.2 0.999999999893022
2.3 0.999999999891605
2.4 0.999999999890187
2.5 0.999999999879197
2.6 0.999999999834831
2.7 0.999999999790466
2.8 0.9999999997461
2.9 0.999999999701734
3 0.999999999657369
3.1 0.999999999613003
3.2 0.999999999568637
3.3 0.999999999479713
3.4 0.999999999352813
3.5 0.999999999261376
3.6 0.999999999205403
3.7 0.999999999184894
3.8 0.999999999199848
3.9 0.999999999250265
4 0.999999999309138
4.1 0.999999999116916
4.2 0.999999999057737
4.3 0.999999999131599
4.4 0.999999999338504
4.5 0.999999999678452
4.6 1.00000000015144
4.7 1.00000000075747
4.8 1.00000000134576
4.9 1.00000000198017
5 1.00000000292881
5.1 1.00000000431836
5.2 1.00000000634907
5.3 1.000000009292
5.4 1.00000001350644
5.5 1.000000019662
5.6 1.0000000292985
5.7 1.00000004425913
5.8 1.0000000667924
5.9 1.00000010079717
6 1.00000015142002
6.1 1.00000022605992
6.2 1.0000003360225
6.3 1.00000049970556
6.4 1.00000074395531
6.5 1.00000111069266
6.6 1.00000165555603
6.7 1.00000246708924
6.8 1.00000367944972
6.9 1.00000549052077
7 1.00000819061147
7.1 1.00001221892438
7.2 1.00001823241713
7.3 1.00002720029182
7.4 1.00004057828433
7.5 1.00006053581104
7.6 1.00009030789516
7.7 1.00013472205
7.8 1.0002009866307
7.9 1.00029984449098
8 1.0004473321498
8.1 1.000667378079
8.2 1.00099569432795
8.3 1.0014855840138
8.4 1.0022166369285
8.5 1.00330773841745
8.6 1.00493657219476
8.7 1.00736897280616
8.8 1.01100318456327
8.9 1.0164370588275
9 1.02457086455711
9.1 1.03676649776543
9.2 1.05509824728417
9.3 1.08275757332901
9.4 1.12472897699509
9.5 1.18897186948062
9.6 1.28861905356837
9.7 1.44641106725058
9.8 1.70464430973126
9.9 2.1509221472866
10 3
};
\end{axis}

\end{tikzpicture}\\[.1in]
      % The case of $F\perp C$:\\[.2in]
      % This file was created by tikzplotlib v0.9.4.
\begin{tikzpicture}

\definecolor{color0}{rgb}{0,0.419607843137255,0.643137254901961}

\begin{axis}[
tick align=outside,
tick pos=left,
title={DRE solution \(\displaystyle (\|P\|\))},
width=\textwidth,
x grid style={white!69.0196078431373!black},
xlabel={\(\displaystyle t\)},
xmin=-0.5, xmax=10.5,
xtick style={color=black},
y grid style={white!69.0196078431373!black},
ylabel={\(\displaystyle F \perp C\):},
ymin=2.76602308433269, ymax=7.91351522901344,
ytick style={color=black}
]
\addplot [semithick, color0, dash pattern=on 3pt off 1pt on 2pt off 1pt]
table {%
0 7.67953830961126
0.1 7.67953830961707
0.2 7.67953830962289
0.3 7.6795383096287
0.4 7.67953830963452
0.5 7.67953830964034
0.6 7.67953830964615
0.7 7.67953830965197
0.8 7.67953830965779
0.9 7.67953830977665
1 7.67953830995972
1.1 7.67953831014278
1.2 7.67953831032585
1.3 7.67953831050891
1.4 7.67953831069198
1.5 7.67953831087504
1.6 7.6795383110581
1.7 7.67953831132597
1.8 7.67953831188892
1.9 7.67953831231655
2 7.67953831260887
2.1 7.67953831276587
2.2 7.67953831278757
2.3 7.67953831267395
2.4 7.67953831242501
2.5 7.67953831241035
2.6 7.67953831313129
2.7 7.67953831334613
2.8 7.6795383130549
2.9 7.67953831225758
3 7.67953831095417
3.1 7.67953830914468
3.2 7.67953830682911
3.3 7.67953830408245
3.4 7.6795383018876
3.5 7.67953829878829
3.6 7.67953829430557
3.7 7.67953828792268
3.8 7.67953827893745
3.9 7.67953826611891
4 7.67953824817779
4.1 7.67953822267346
4.2 7.67953818422167
4.3 7.67953812741553
4.4 7.67953804272939
4.5 7.67953792680743
4.6 7.67953774649642
4.7 7.67953749535083
4.8 7.67953714244923
4.9 7.67953663580898
5 7.67953590680308
5.1 7.67953486239499
5.2 7.67953335777513
5.3 7.67953120928732
5.4 7.67952812311669
5.5 7.67952369612282
5.6 7.6795173852114
5.7 7.67950837549463
5.8 7.67949552108489
5.9 7.67947720438723
6 7.67945113566588
6.1 7.67941406803217
6.2 7.67936143233218
6.3 7.6792867984752
6.4 7.67918106920625
6.5 7.67903151427185
6.6 7.67882027944573
6.7 7.67852238039375
6.8 7.67810296340423
6.9 7.6775134710273
7 7.67668642363281
7.1 7.67552844107023
7.2 7.67391056223093
7.3 7.67165535703501
7.4 7.66851965912327
7.5 7.66417166326487
7.6 7.65816092278309
7.7 7.64987954996281
7.8 7.63851283388325
7.9 7.62297790948845
8 7.60184985108048
8.1 7.57327671298219
8.2 7.53488817182394
8.3 7.48370914938861
8.4 7.41609781026491
8.5 7.32774042329675
8.6 7.21374915524528
8.7 7.06892053758574
8.8 6.88821237915074
8.9 6.66747301064804
9 6.40439539683488
9.1 6.09956863821748
9.2 5.75738454901152
9.3 5.38648116710737
9.4 4.9994248640674
9.5 4.61145168665294
9.6 4.23820469215689
9.7 3.89236592760722
9.8 3.57894135176656
9.9 3.28956882145407
10 3
};
\end{axis}

\end{tikzpicture}% \\[.1in]
  \end{minipage}
  \begin{minipage}[t]{0.3\linewidth}
    \begin{center}
      % \hlinpic
      % This file was created by tikzplotlib v0.9.4.
\begin{tikzpicture}

\definecolor{color0}{rgb}{0,0.419607843137255,0.643137254901961}
\definecolor{color1}{rgb}{1,0.501960784313725,0.0549019607843137}

\begin{axis}[
log basis y={10},
tick align=outside,
tick pos=left,
title={State (log of \(\displaystyle |x|\))},
width=\textwidth,
x grid style={white!69.0196078431373!black},
xlabel={\(\displaystyle t\)},
xmin=-0.5, xmax=10.5,
xtick style={color=black},
y grid style={white!69.0196078431373!black},
ymin=2.6555218381979e-06, ymax=201140615.301002,
ymode=log,
ytick style={color=black}
]
\addplot [semithick, color0, dash pattern=on 3pt off 1pt on 2pt off 1pt]
table {%
0 1
0.1 0.900000000029195
0.2 0.950944671228052
0.3 1.07108215597844
0.4 1.24193536357071
0.5 1.4611409291046
0.6 1.73260626775845
0.7 2.06378589455063
0.8 2.46491198749269
0.9 2.94892100873981
1 3.53169981822434
1.1 4.23252684588309
1.2 5.07467031722162
1.3 6.08614064883476
1.4 7.30061155154291
1.5 8.75853543337537
1.6 10.5084877116412
1.7 12.6087835140362
1.8 15.1294199537324
1.9 18.1544083391805
2 21.7845738474906
2.1 26.1409158638482
2.2 31.3686409281583
2.3 37.6420026780845
2.4 45.170110093308
2.5 54.2038976317484
2.6 65.0444895837869
2.7 78.0532374479303
2.8 93.6637649010209
2.9 112.396421857443
3 134.875629415825
3.1 161.85069385545
3.2 194.220783480146
3.3 233.064900869475
3.4 279.67784961096
3.5 335.613394403374
3.6 402.736053200394
3.7 483.283247798746
3.8 579.939884556607
3.9 695.927851265778
4 835.113413406415
4.1 1002.13608965918
4.2 1202.56330252508
4.3 1443.07595907344
4.4 1731.69114784291
4.5 2078.02937512547
4.6 2493.63524850943
4.7 2992.362297133
4.8 3590.8347559902
4.9 4309.00170709928
5 5170.80204890618
5.1 6204.9624595698
5.2 7445.9549529056
5.3 8935.14594551913
5.4 10722.1751373675
5.5 12866.6101684351
5.6 15439.9322067452
5.7 18527.9186539776
5.8 22233.5023922109
5.9 26680.2028800172
6 32016.2434677792
6.1 38419.4921760762
6.2 46103.390629752
6.3 55324.0687788056
6.4 66388.8825634676
6.5 79666.6591123047
6.6 95599.9909799592
6.7 114719.989232454
6.8 137663.987149583
6.9 165196.784667804
7 198236.141711752
7.1 237883.370192091
7.2 285460.044402998
7.3 342552.053499212
7.4 411062.464468577
7.5 493274.957699199
7.6 591929.949660177
7.7 710315.940118643
7.8 852379.128800425
7.9 1022854.9553831
8 1227425.947488
8.1 1472911.13827104
8.2 1767493.36753221
8.3 2120992.04304764
8.4 2545190.45416894
8.5 3054228.54814346
8.6 3665074.26169996
8.7 4398089.11895325
8.8 5277706.94889211
8.9 6333248.34636813
9 7599898.02528682
9.1 9119877.64244378
9.2 10943853.1861383
9.3 13132623.8425267
9.4 15759148.6352744
9.5 18910978.3931904
9.6 22693174.1114898
9.7 27231808.9855191
9.8 32678170.8517176
9.9 39213805.1180763
10 47056566.2849422
};
\addplot [semithick, color1, dash pattern=on 5pt off 2pt]
table {%
0 1
0.1 0.600000000029195
0.2 0.410944671219294
0.3 0.299798754602146
0.4 0.226455655538505
0.5 0.174628582804408
0.6 0.136402877356893
0.7 0.107420962861693
0.8 0.0850477806074577
0.9 0.0675696262952944
1 0.053807271402336
1.1 0.0429136082759876
1.2 0.0342603496102921
1.3 0.0273705828180813
1.4 0.0218762974774726
1.5 0.0174902392536118
1.6 0.0139864069190321
1.7 0.0111860262938353
1.8 0.00894716055340551
1.9 0.00715683919967455
2 0.00572499575379443
2.1 0.00457974303783782
2.2 0.00366366027453336
2.3 0.00293085854162211
2.4 0.00234465229400669
2.5 0.00187570684343912
2.6 0.0015005618478767
2.7 0.00120045304911496
2.8 0.000960371248819815
2.9 0.000768310341560891
3 0.000614666201323566
3.1 0.000491756041193959
3.2 0.000393434042435878
3.3 0.000314783938537107
3.4 0.000251873134857778
3.5 0.000201556043297333
3.6 0.000161316783638827
3.7 0.000129143379144358
3.8 0.000103427152155289
3.9 8.28822873652928e-05
4 6.648153941414e-05
4.1 5.34048697690571e-05
4.2 4.29984443324204e-05
4.3 3.47419414980497e-05
4.4 2.82225359437663e-05
4.5 2.31142572961451e-05
4.6 1.91616915712167e-05
4.7 1.61672104557934e-05
4.8 1.39810898805872e-05
4.9 1.24940238095321e-05
5 1.16316589350592e-05
5.1 1.13508770114256e-05
5.2 1.16376389417542e-05
5.3 1.25062828229506e-05
5.4 1.40002408509452e-05
5.5 1.61942109308907e-05
5.6 1.91978915735091e-05
5.7 2.31614668239995e-05
5.8 2.82831154700764e-05
5.9 3.48189199916032e-05
6 4.30956707044791e-05
6.1 5.35272053122942e-05
6.2 6.66351008589207e-05
6.3 8.30747526990926e-05
6.4 0.000103668144498478
6.5 0.00012944494784777
6.6 0.000161694006660404
6.7 0.00020202778084376
6.8 0.000252462972080765
6.9 0.000315521364145201
7 0.000394355921830258
7.1 0.000492908457256171
7.2 0.000616106754281818
7.3 0.000770111020830458
7.4 0.000962622016642725
7.5 0.00120326631043818
7.6 0.00150407801700282
7.7 0.00188010126490246
7.8 0.00235014382396152
7.9 0.00293772014477866
8 0.00367223202423193
8.1 0.00459044791601854
8.2 0.00573835855401277
8.3 0.00717350858232629
8.4 0.0089679336821522
8.5 0.0112118741825811
8.6 0.0140184959871573
8.7 0.0175299394858193
8.8 0.0219251579306034
8.9 0.027430237111179
9 0.0343322811207626
9.1 0.0429986460419478
9.2 0.0539045876656058
9.3 0.0676748495418113
9.4 0.0851496404645815
9.5 0.107495831289393
9.6 0.1364075946297
9.7 0.174498192511811
9.8 0.226142999821913
9.9 0.299543908118252
10 0.412840207253802
};
\end{axis}

\end{tikzpicture}\\[.1in]
      % \hlinpic
      % This file was created by tikzplotlib v0.9.4.
\begin{tikzpicture}

\definecolor{color0}{rgb}{0,0.419607843137255,0.643137254901961}
\definecolor{color1}{rgb}{1,0.501960784313725,0.0549019607843137}

\begin{axis}[
log basis y={10},
tick align=outside,
tick pos=left,
title={State (log of \(\displaystyle |x|\))},
width=\textwidth,
x grid style={white!69.0196078431373!black},
xlabel={\(\displaystyle t\)},
xmin=-0.5, xmax=10.5,
xtick style={color=black},
y grid style={white!69.0196078431373!black},
ymin=4.86120453067125e-06, ymax=1.7906561371592,
ymode=log,
ytick style={color=black}
]
\addplot [semithick, color0, dash pattern=on 3pt off 1pt on 2pt off 1pt]
table {%
0 1
0.1 0.900000000229206
0.2 0.774969354968253
0.3 0.659104805646576
0.4 0.56097349125966
0.5 0.477849744609044
0.6 0.404712661651945
0.7 0.337734715169625
0.8 0.275029797016837
0.9 0.216339566266657
1 0.162392011998686
1.1 0.114254428814015
1.2 0.0728386350158506
1.3 0.038620719724444
1.4 0.0115635647237431
1.5 0.00881651225342273
1.6 0.0233110245018779
1.7 0.0328680803840352
1.8 0.0384648909592065
1.9 0.0410183728005278
2 0.0413334704507078
2.1 0.0400816020938295
2.2 0.0377997693037643
2.3 0.0349018921170686
2.4 0.0316960244619997
2.5 0.0284032288407493
2.6 0.0251756191578832
2.7 0.0221123099045053
2.8 0.0192728038225085
2.9 0.0166878143322999
3 0.014367760074731
3.1 0.0123092689481922
3.2 0.0105000475965925
3.3 0.00892244815477914
3.4 0.007556020981359
3.5 0.00637929345820264
3.6 0.0053709687965201
3.7 0.00451069756795107
3.8 0.0037795400647977
3.9 0.00316020944881992
4 0.0026371631660924
4.1 0.00219659253076062
4.2 0.00182634691284538
4.3 0.00151581863830368
4.4 0.00125580699841609
4.5 0.00103837402438661
4.6 0.000856700449628337
4.7 0.000704947212125493
4.8 0.000578125644255169
4.9 0.000471977941347771
5 0.000382868407367421
5.1 0.000307685238362431
5.2 0.00024375210761174
5.3 0.000188748486035143
5.4 0.000140637419871614
5.5 9.75993436359182e-05
5.6 5.79704017417401e-05
5.7 2.01836600585405e-05
5.8 1.72885107292605e-05
5.9 5.59927238621961e-05
6 9.755554399748e-05
6.1 0.000143749359897919
6.2 0.000196565675781369
6.3 0.000258298707180574
6.4 0.000331642640727499
6.5 0.000419806511213323
6.6 0.000526651386077055
6.7 0.000656855446172932
6.8 0.000816113637192703
6.9 0.00101137985561834
7 0.00125116113517219
7.1 0.00154587500379266
7.2 0.00190828304475718
7.3 0.00235401559219274
7.4 0.00290220421524762
7.5 0.00357623978667534
7.6 0.00440467382281767
7.7 0.00542227835309632
7.8 0.00667127319268649
7.9 0.00820271681308879
8 0.01007803496859
8.1 0.0123706263287231
8.2 0.0151674337941826
8.3 0.0185703043611612
8.4 0.0226968887036965
8.5 0.0276807807687434
8.6 0.0336706255112968
8.7 0.0408281303837942
8.8 0.0493254486009821
8.9 0.0593434145917527
9 0.0710736739196651
9.1 0.0847296806799055
9.2 0.100573229169854
9.3 0.118963524980841
9.4 0.140433132732771
9.5 0.165786734079165
9.6 0.1961980989312
9.7 0.233231140164514
9.8 0.278595238416338
9.9 0.333219791291828
10 0.395015832938627
};
\addplot [semithick, color1, dash pattern=on 5pt off 2pt]
table {%
0 1
0.1 0.600000000229206
0.2 0.234969354899491
0.3 0.0593860009057852
0.4 0.283399676331438
0.5 0.450378153600843
0.6 0.574047370119666
0.7 0.664563111920408
0.8 0.728358661915081
0.9 0.76921898587712
1 0.78951255481071
1.1 0.791177284914047
1.2 0.77632623598361
1.3 0.747479254679826
1.4 0.707512628157433
1.5 0.659454155263604
1.6 0.606239949535014
1.7 0.550510805563294
1.8 0.494482919505329
1.9 0.439895131204276
2 0.388017041173923
2.1 0.339696774609511
2.2 0.295428943939729
2.3 0.255428218498307
2.4 0.219699150569994
2.5 0.188097234999344
2.6 0.160379256048394
2.7 0.1362428973586
2.8 0.115356639559842
2.9 0.0973814253491475
3 0.0819856656902039
3.1 0.0688550559796984
3.2 0.0576984752404905
3.3 0.0482510187553096
3.4 0.0402750000754026
3.5 0.0335595683484342
3.6 0.0279194281602677
3.7 0.0231930203563679
3.8 0.0192404213039875
3.9 0.0159411405446514
4 0.0131919383176948
4.1 0.0109047412173751
4.2 0.00900470297157019
4.3 0.0074284350173024
4.4 0.00612241614802383
4.5 0.00504158015950875
4.6 0.00414807376392227
4.7 0.00341017306010153
4.8 0.00280134474379596
4.9 0.00229943743765793
5 0.00188598857164223
5.1 0.00154563286399954
5.2 0.00126559939917641
5.3 0.00103528541615982
5.4 0.00084589611117328
5.5 0.000690140939845934
5.6 0.000561978035239978
5.7 0.000456399409684433
5.8 0.000369250565916481
5.9 0.000297078998337749
6 0.000237006823141129
6.1 0.000186623433726074
6.2 0.0001438946464496
6.3 0.000107085285561709
6.4 7.46925648947273e-05
6.5 4.53879660649309e-05
6.6 1.79655968373706e-05
6.7 8.70474572683259e-06
6.8 3.57213729393337e-05
6.9 6.41927266325e-05
7 9.52787623994303e-05
7.1 0.000130232527745514
7.2 0.000170442315176953
7.3 0.000217474022143387
7.4 0.000273112124545379
7.5 0.000339395640469031
7.6 0.000418642155229397
7.7 0.00051344770542157
7.8 0.000626642103850317
7.9 0.000761166875330291
8 0.000919824980680672
8.1 0.00110482684902777
8.2 0.00131702636383983
8.3 0.00155470753559792
8.4 0.0018117602523412
8.5 0.00207509855141469
8.6 0.00232127728507792
8.7 0.00251252932680822
8.8 0.00259296853455646
8.9 0.00248654795167498
9 0.0020994695660694
9.1 0.00133079458576982
9.2 9.53274811599579e-05
9.3 0.00163855528993999
9.4 0.00379779700518367
9.5 0.00615104250482553
9.6 0.00828192021814135
9.7 0.00966030674925317
9.8 0.00997640585540608
9.9 0.0100732600776428
10 0.0139138506814453
};
\end{axis}

\end{tikzpicture}
    \end{center}
    \vfill
  \end{minipage}
  \begin{minipage}[t]{0.3\linewidth}
    \begin{center}
      % \hlinpic
      % \hlinpic
      % This file was created by tikzplotlib v0.9.4.
\begin{tikzpicture}

\definecolor{color0}{rgb}{0,0.419607843137255,0.643137254901961}
\definecolor{color1}{rgb}{1,0.501960784313725,0.0549019607843137}

\begin{axis}[
tick align=outside,
tick pos=left,
title={Output \(\displaystyle (Fx, Cx)\)},
width=\textwidth,
x grid style={white!69.0196078431373!black},
xlabel={\(\displaystyle t\)},
xmin=-0.5, xmax=10.5,
xtick style={color=black},
y grid style={white!69.0196078431373!black},
ymin=-1.52547975840725, ymax=1.8871713107106,
ytick style={color=black}
]
\addplot [semithick, color0, dash pattern=on 3pt off 1pt on 2pt off 1pt]
table {%
0 1.73205080756888
0.1 1.55884572720899
0.2 1.34228629711389
0.3 1.14160301089268
0.4 0.971634588561026
0.5 0.827660036046676
0.6 0.700982892447601
0.7 0.584973686153594
0.8 0.476365582028516
0.9 0.374711120461264
1 0.281271215525059
1.1 0.197894475695635
1.2 0.126160216601419
1.3 0.0668930487876146
1.4 0.0200286816181343
1.5 -0.0152706471684817
1.6 -0.0403758788137355
1.7 -0.056929185172407
1.8 -0.0666231454489425
1.9 -0.0710459057343154
2 -0.0715916708737728
2.1 -0.0694233712752719
2.2 -0.0654711209485023
2.3 -0.0604518504270505
2.4 -0.0548991247661295
2.5 -0.0491958354511835
2.6 -0.0436054514934581
2.7 -0.0382996442273117
2.8 -0.0333814754248923
2.9 -0.0289041422908195
3 -0.0248856904403937
3.1 -0.0213202792222987
3.2 -0.0181866159191898
3.3 -0.0154541335319767
3.4 -0.0130874122427702
3.5 -0.0110492603859987
3.6 -0.00930279084143989
3.7 -0.00781275736526862
3.8 -0.00654635542147179
3.9 -0.00547364332791534
4 -0.00456770059152123
4.1 -0.0038046098668037
4.2 -0.00316332564529476
4.3 -0.00262547489660185
4.4 -0.00217512152575724
4.5 -0.00179851656749738
4.6 -0.00148384870562338
4.7 -0.00122100438805539
4.8 -0.00100134298900844
4.9 -0.000817489774466102
5 -0.000663147534173352
5.1 -0.000532926465582672
5.2 -0.000422191034835531
5.3 -0.000326921967664572
5.4 -0.000243591156663031
5.5 -0.000169047021962784
5.6 -0.000100407681151873
5.7 -3.49591247040908e-05
5.8 2.99445789702789e-05
5.9 9.6982242583498e-05
6 0.000168971158763656
6.1 0.000248981194898699
6.2 0.000340461737477443
6.3 0.00044738648436611
6.4 0.00057442190369634
6.5 0.000727126206769709
6.6 0.000912186958562031
6.7 0.00113770700599984
6.8 0.00141355028436759
6.9 0.00175176129568265
7 0.00216707465457379
7.1 0.00267753404871961
7.2 0.00330524318874167
7.3 0.00407727460748717
7.4 0.00502676515474944
7.5 0.00619422901057097
7.6 0.00762911885188885
7.7 0.00939166160034373
7.8 0.0115549841209052
7.9 0.0142075222803693
8 0.0174556686060536
8.1 0.0214265533227977
8.2 0.0262707659519614
8.3 0.032164710665549
8.4 0.0393121644085384
8.5 0.047944518684639
8.6 0.0583192341081908
8.7 0.0707163962027782
8.8 0.0854341830830282
8.9 0.10278580916754
9 0.123103214309443
9.1 0.146756111846683
9.2 0.174197942803455
9.3 0.206050869514305
9.4 0.243237320959223
9.5 0.287151046646024
9.6 0.339825075697263
9.7 0.403968184672156
9.8 0.482541107683862
9.9 0.577153608604944
10 0.684187492443842
};
\addplot [semithick, color1, dash pattern=on 5pt off 2pt]
table {%
0 1.73205080756888
0.1 1.03923048493832
0.2 0.406978860907601
0.3 -0.102859570827151
0.4 -0.490862638254626
0.5 -0.780077844655719
0.6 -0.994279210998558
0.7 -1.15105707468223
0.8 -1.2615542085698
0.9 -1.33232636568578
1 -1.36747585814566
1.1 -1.37035925526553
1.2 -1.34463648397232
1.3 -1.29467204670917
1.4 -1.22544781896526
1.5 -1.14220810217898
1.6 -1.05003839417264
1.7 -0.953512685351297
1.8 -0.856469540058222
1.9 -0.761920717247984
2 -0.67206522951578
2.1 -0.588372072790946
2.2 -0.511697940930028
2.3 -0.442414652125873
2.4 -0.380530091166954
2.5 -0.325793967782087
2.6 -0.277785019955916
2.7 -0.235979620395486
2.8 -0.199803560708056
2.9 -0.168669576418199
3 -0.14200333846779
3.1 -0.119260455314837
3.2 -0.0999366906357845
3.3 -0.083573216001155
3.4 -0.0697583464054377
3.5 -0.0581268774595684
3.6 -0.048357868091853
3.7 -0.0401714896382084
3.8 -0.033325387257537
3.9 -0.0276108653539325
4 -0.0228491074165621
4.1 -0.0188875658318841
4.2 -0.015596603053826
4.3 -0.0128664268706915
4.4 -0.0106043358334574
4.5 -0.00873227298670035
4.6 -0.00718467451265685
4.7 -0.00590659300269849
4.8 -0.00485207142577062
4.9 -0.00398274247084953
5 -0.0032666280285786
5.1 -0.00267711465029539
5.2 -0.00219208246140218
5.3 -0.0017931669411239
5.4 -0.00146513504247705
5.5 -0.00119535917219649
5.6 -0.000973374509773376
5.7 -0.000790506966117881
5.8 -0.000639560740890905
5.9 -0.000514555918982651
6 -0.000410507859420927
6.1 -0.000323241269096523
6.2 -0.000249232838587867
6.3 -0.000185477155335901
6.4 -0.000129371317345303
6.5 -7.86142632766724e-05
6.6 -3.11173265106246e-05
6.7 1.50770618658421e-05
6.8 6.18712328470419e-05
6.9 0.00011118506400387
7 0.000165027657358096
7.1 0.000225569354853354
7.2 0.00029521474964615
7.3 0.000376676055678705
7.4 0.000473044075875675
7.5 0.000587850493159741
7.6 0.000725109483047452
7.7 0.000889317512819818
7.8 0.0010853759620306
7.9 0.00131837970111051
8 0.00159318360060998
8.1 0.00191361623608232
8.2 0.00228115657707828
8.3 0.00269283244256579
8.4 0.00313806080818877
8.5 0.00359417612176283
8.6 0.0040205701962105
8.7 0.00435182844953866
8.8 0.00449115324427921
8.9 0.00430682738775739
9 0.00363638795737678
9.1 0.00230500383699091
9.2 0.000165112040726612
9.3 -0.00283806101318682
9.4 -0.00657797736981103
9.5 -0.0106539181378735
9.6 -0.0143447066020527
9.7 -0.016732142106407
9.8 -0.017279641818491
9.9 -0.0174473982523326
10 -0.0240994963091901
};
\end{axis}

\end{tikzpicture}\\[.1in]
      % This file was created by tikzplotlib v0.9.4.
\begin{tikzpicture}

\definecolor{color0}{rgb}{0,0.419607843137255,0.643137254901961}
\definecolor{color1}{rgb}{1,0.501960784313725,0.0549019607843137}

\begin{axis}[
tick align=outside,
tick pos=left,
title={Output \(\displaystyle (Fx, Cx)\)},
width=\textwidth,
x grid style={white!69.0196078431373!black},
xlabel={\(\displaystyle t\)},
xmin=-0.5, xmax=10.5,
xtick style={color=black},
y grid style={white!69.0196078431373!black},
ymin=-0.0865818970679649, ymax=1.81865236493254,
ytick style={color=black}
]
\addplot [semithick, color0, dash pattern=on 3pt off 1pt on 2pt off 1pt]
table {%
0 1.73205080756888
0.1 1.03923048459189
0.2 0.711777049651505
0.3 0.519266675016791
0.4 0.392232701054007
0.5 0.302465577870984
0.6 0.236256713880725
0.7 0.186058565474422
0.8 0.147307077083088
0.9 0.117034025791892
1 0.0931969278854938
1.1 0.0743285498701187
1.2 0.0593406662100984
1.3 0.0474072400736885
1.4 0.0378908587124735
1.5 0.0302939830237913
1.6 0.0242251673990965
1.7 0.0193747658757242
1.8 0.0154969366619744
1.9 0.0123960091154369
2 0.00991598351868804
2.1 0.00793234762714493
2.2 0.00634564573716351
2.3 0.00507639590388671
2.4 0.00406105689930251
2.5 0.0032488195529412
2.6 0.00259904936042189
2.7 0.0020792456731681
2.8 0.00166341179708429
2.9 0.00133075254756406
3 0.00106463309038778
3.1 0.000851746448276871
3.2 0.000681447750926151
3.3 0.000545221774952907
3.4 0.000436257066635318
3.5 0.000349105307563533
3.6 0.000279408865376044
3.7 0.000223682894139159
3.8 0.000179141082415117
3.9 0.000143556332764211
4 0.000115149404030683
4.1 9.24999478116061e-05
4.2 7.44754902301741e-05
4.3 6.01748078282077e-05
4.4 4.88828661730421e-05
4.5 4.00350680161429e-05
4.6 3.31890233603116e-05
4.7 2.8002429926093e-05
4.8 2.42159580183642e-05
4.9 2.16402840290849e-05
5 2.0146624251835e-05
5.1 1.96602956942547e-05
5.2 2.01569819272604e-05
5.3 2.16615172631763e-05
5.4 2.42491284720385e-05
5.5 2.804919612079e-05
5.6 3.32517236035162e-05
5.7 4.01168373169881e-05
5.8 4.89877929905097e-05
5.9 6.03081384901325e-05
6 7.46438912464154e-05
6.1 9.27118391880643e-05
6.2 0.000115415380255127
6.3 0.000143889692501048
6.4 0.000179558493397756
6.5 0.000224205226455442
6.6 0.000280062234815199
6.7 0.000349922380961783
6.8 0.000437278694673727
6.9 0.00054649903357293
7 0.000683044492875668
7.1 0.00085374249144808
7.2 0.00106712820130246
7.3 0.00133387141554709
7.4 0.00166731024130961
7.5 0.00208411838471486
7.6 0.00260513954399634
7.7 0.00325643091418558
7.8 0.00407056850819556
7.9 0.00508828054917524
8 0.00636049244315122
8.1 0.00795088902004278
8.2 0.00993912856759759
8.3 0.0124248813331205
8.4 0.0155329167763959
8.5 0.0194195357323002
8.6 0.0242807472954569
8.7 0.0303627458430468
8.8 0.0379754874997767
8.9 0.0475105643402234
9 0.0594652552408985
9.1 0.074475839601324
9.2 0.0933654845978799
9.3 0.117216277800997
9.4 0.147483503530878
9.5 0.186188241395081
9.6 0.236264884436899
9.7 0.302239735259392
9.8 0.391691165467592
9.9 0.518825267958555
10 0.71506021437085
};
\addplot [semithick, color1, dash pattern=on 5pt off 2pt]
table {%
0 1.73205080756888
0.1 1.03923048459189
0.2 0.711777049651505
0.3 0.519266675016791
0.4 0.392232701054007
0.5 0.302465577870984
0.6 0.236256713880725
0.7 0.186058565474422
0.8 0.147307077083088
0.9 0.117034025791892
1 0.0931969278854938
1.1 0.0743285498701187
1.2 0.0593406662100984
1.3 0.0474072400736885
1.4 0.0378908587124735
1.5 0.0302939830237913
1.6 0.0242251673990965
1.7 0.0193747658757242
1.8 0.0154969366619744
1.9 0.0123960091154369
2 0.00991598351868804
2.1 0.00793234762714493
2.2 0.00634564573716351
2.3 0.00507639590388671
2.4 0.00406105689930251
2.5 0.0032488195529412
2.6 0.00259904936042189
2.7 0.0020792456731681
2.8 0.00166341179708429
2.9 0.00133075254756406
3 0.00106463309038778
3.1 0.000851746448276871
3.2 0.000681447750926151
3.3 0.000545221774952907
3.4 0.000436257066635318
3.5 0.000349105307563533
3.6 0.000279408865376044
3.7 0.000223682894139159
3.8 0.000179141082415117
3.9 0.000143556332764211
4 0.000115149404030683
4.1 9.24999478116061e-05
4.2 7.44754902301741e-05
4.3 6.01748078282077e-05
4.4 4.88828661730421e-05
4.5 4.00350680161429e-05
4.6 3.31890233603116e-05
4.7 2.8002429926093e-05
4.8 2.42159580183642e-05
4.9 2.16402840290849e-05
5 2.0146624251835e-05
5.1 1.96602956942547e-05
5.2 2.01569819272604e-05
5.3 2.16615172631763e-05
5.4 2.42491284720385e-05
5.5 2.804919612079e-05
5.6 3.32517236035162e-05
5.7 4.01168373169881e-05
5.8 4.89877929905097e-05
5.9 6.03081384901325e-05
6 7.46438912464154e-05
6.1 9.27118391880643e-05
6.2 0.000115415380255127
6.3 0.000143889692501048
6.4 0.000179558493397756
6.5 0.000224205226455442
6.6 0.000280062234815199
6.7 0.000349922380961783
6.8 0.000437278694673727
6.9 0.00054649903357293
7 0.000683044492875668
7.1 0.00085374249144808
7.2 0.00106712820130246
7.3 0.00133387141554709
7.4 0.00166731024130961
7.5 0.00208411838471486
7.6 0.00260513954399634
7.7 0.00325643091418558
7.8 0.00407056850819556
7.9 0.00508828054917524
8 0.00636049244315122
8.1 0.00795088902004278
8.2 0.00993912856759759
8.3 0.0124248813331205
8.4 0.0155329167763959
8.5 0.0194195357323002
8.6 0.0242807472954569
8.7 0.0303627458430468
8.8 0.0379754874997767
8.9 0.0475105643402234
9 0.0594652552408985
9.1 0.074475839601324
9.2 0.0933654845978799
9.3 0.117216277800997
9.4 0.147483503530878
9.5 0.186188241395081
9.6 0.236264884436899
9.7 0.302239735259392
9.8 0.391691165467592
9.9 0.518825267958555
10 0.71506021437085
};
\end{axis}

\end{tikzpicture}
    \end{center}
    \vfill
  \end{minipage}
  \caption{Example simulation of the optimal control problem (Problem
    \ref{prb-fiti-optcont}) with coefficients as in \eqref{eqn:example-abc}, with
    the initial value $x(0)= [1 ~ 1]^*$ and choices of
    the endpoint constraint $F$ that illustrate the sufficiency and necessity of
    Assumption \ref{ass:szero-invertible} for the turnpike property as in
  Definition \ref{def:turnpike}.}
  \label{fig:turnpikenno}
\end{figure}

\section{Explicit Formulas and Turnpike Estimates for Solutions to the Affine
Problem}\label{sec:ode-lqr-affine}

In this section, we use an explicit formula of the state transition
matrices to derive formulas for the solution to the finite time optimal control
problem.

\begin{lemma}\label{lem:fundamental-solution}
  Under Assumption \ref{ass:exist-are-sol-stab}, the fundamental solution matrix
  $U$ to
  \begin{equation*}
    \dot U = (A-BB^*P(t))U, \quad U(t_1)=I,
  \end{equation*}
  where $P$ solves the DRE with $P(t_1)=S$ is given as
  \begin{equation}\label{eqn:explct-U}
    U(t) =
    \mxp{-(t_1-t)\AP}\bigl(I-[W-\mxp{(t_1-t)\AP}W\mxp{(t_1-t)\APs}](\PP-S)\bigr).
  \end{equation}
\end{lemma}

\begin{proof}
  This formula has been used in the literature in a more or less explicit way.
  A direct derivation is provided in \cite[Proof of Thm. 3.4]{BehBH19a}.
\end{proof}

\begin{corollary}[of Lemma
  \ref{lem:fundamental-solution}]\label{cor:state-transition-maps}
Given initial conditions $\alpha$, $\beta$ and an inhomogeneity $f$. With $U$ as
in \eqref{eqn:explct-U}, the state transition for the forward evolution of
\begin{equation*}
  \dot x(t) = (A-BB^*P(t))x + f(t), \quad x(t_0)=\alpha
\end{equation*}
is given as
\begin{equation*}
  x(t) = U(t)U(t_0)^{-1}\alpha + \int_{t_0}^{t} U(t)U(s)^{-1}f(s)ds
\end{equation*}
and the backwards propagation of
\begin{equation*}
  -\dot y(t) = (A^*-P(t)BB^*)y(t)+f(t), \quad y(t_0)=\beta
\end{equation*}
is given as
\begin{equation*}
  y(t) = U(t)^{-*}U(t_0)^*\beta - \int_{t_0}^{t} U(t)^{-*}U(s)^*f(s)ds.
\end{equation*}
\end{corollary}

The following technical lemma provides a explicit expressions of the forward and
backward state transition maps.

\begin{lemma}
  Let Assumption \ref{ass:exist-are-sol-stab} hold. Then the state transition
  maps used in Corollary \ref{cor:state-transition-maps} fulfill the
  relations
\begin{align}
  U(t)U(s)^{-1} &= \mxp{(t-s)\AP}-W(t-s)\mxp{(t_1-t)\APs}\tilde
  S(t_1-s)\mxp{(t_1-s)\AP} \label{eqn:utusmo}\\
  \intertext{and}
  U(t)^{-*} U(s)^* &= \mxp{(s-t)\APs} - \mxp{(t_1-t)\APs}\tilde
  S(t_1-t)\mxp{(t_1-s)\AP}W(s-t),\label{eqn:utmsuss}
\end{align}
with the definition of $W$ and $\tilde S$ as in \eqref{eqn:def-reach-gram} and
\eqref{eqn:def-tilde-s}.
Furthermore, in the case that $s=t_1$, relation \eqref{eqn:utmsuss} reduces to
\begin{equation*}
  U(t)^{-*}U(t_1)^*=\mxp{(t_1-t)\APs}[I-\tilde S(t_1-t)W(t_1-t)].\label{eqn:utmsuss-toneiss}
\end{equation*}
\end{lemma}

\begin{proof}
Using that $\PP$, $S$, and $W$ are symmetric, we write 
\begin{equation*}
  U(t) = \mxp{-(t_1-t)\AP}(I+W(t_1-t)(S-\PP))
\end{equation*}
and
\begin{equation*}
U(t)^{-*} = \mxp{(t_1-t)\APs}(I+W(t_1-t)(S-\PP))^{-1}.
\end{equation*}

Then, with the help of the identity
\begin{align*}
  (I+\gamma_t)(I+\gamma_s)^{-1}&=(I+\gamma_t)[I-\gamma_s(I+\gamma_s)^{-1}] \\
                               &=I+\gamma_t-(I+\gamma_t)\gamma_s(I+\gamma_s)^{-1}\\
                               &=I+[\gamma_t(I+\gamma_s)-(I+\gamma_t)\gamma_s](I+\gamma_s)^{-1}\\
                               &=I+[\gamma_t-\gamma_s](I+\gamma_s)^{-1}\\
                               &=I-[\gamma_s-\gamma_t](I+\gamma_s)^{-1},
\end{align*}
we compute 
\begin{align*}
  U(t)U(s)^{-1} &=
  \mxp{-(t_1-t)\AP}(I+W(t_1-t)(S-\PP))\times \\
                &\quad\quad\quad\quad\quad\quad\quad\quad(I+W(t_1-s)(S-\PP))^{-1}\mxp{(t_1-s)\AP}\\
                &= \mxp{-(t_1-t)\AP}\bigl
              [I-[W(t_1-s)-W(t_1-t)]\tilde S(t_1-s)\bigr ]\mxp{(t_1-s)\AP}\\
                &= \mxp{-(t_1-t)\AP}\bigl
              % [I-[\mxp{(t_1-t)\AP}(W-\mxp{(t-s)\AP}W\mxp{(t-s)\APs})\mxp{(t_1-t)\APs}]\tilde S(t_1-s)\bigr ]\mxp{-(t_1-s)\APs}
              [I-\mxp{(t_1-t)\AP}W(t-s)\mxp{(t_1-t)\APs}\tilde S(t_1-s)\bigr ]\mxp{(t_1-s)\AP} \\
                &= \mxp{(t-s)\AP}-W(t-s)\mxp{(t_1-t)\APs}\tilde S(t_1-s)\mxp{(t_1-s)\AP}
\end{align*}
and
\begin{equation*}
  \begin{split}
  U(t)^{-*} U(s)^* &= \mxp{(t_1-t)\APs}\bigl
              [I-\tilde S(t_1-t)\mxp{(t_1-s)\AP}W(s-t)\mxp{(t_1-s)\APs}\bigr
              ]\mxp{-(t_1-s)\APs} \\
                   &= \mxp{(s-t)\APs} - \mxp{(t_1-t)\APs}\tilde
                   S(t_1-t)\mxp{(t_1-s)\AP}W(s-t).
\end{split}
\end{equation*}
The relation for $s=t_1$ follows by directly from \eqref{eqn:utmsuss}.
\end{proof}

Now we state explicit formulas for the feedforward $w$ as defined in
\eqref{eqn:w-feedforward}.
\begin{lemma}\label{lem:formula-w}
  Let Assumption \ref{ass:exist-are-sol-stab} hold. Then $w$ as defined in
  \eqref{eqn:w-feedforward} has the representation $w=w_h+w_p$, with
\begin{align}
  w_h(t) =& -\mxp{(t_1-t)\APs}[I-\tilde S(t_1-t)W(t_1-t)]F^*\ye,
  \label{eqn:w-homo}  \\[.2in]
  % \intertext{~}
\begin{split}\label{eqn:w-partic}
  w_p(t) 
  =& \APms C^*\yc - \mxp{(t_1-t)\APs}C^*\yc \\ 
   &-\mxp{(t_1-t)\APs}\tilde S(t_1-t)\biggl[ \APmo \bigl[I -
     \mxp{(t_1-t)\AP}\bigr ]W +\\
   &\quad\quad\quad\quad\quad\quad\quad\quad\quad\quad\quad\quad\mxp{(t_1-t)\AP}W\bigl[I -\mxp{(t_1-t)\APs}\bigr]\APms
   \biggr]C^*\yc,
\end{split}
\end{align}
and the definition of $W$ and $\tilde S$ as in \eqref{eqn:def-reach-gram} and
\eqref{eqn:def-tilde-s}.
\end{lemma}

\begin{proof}
  The application of Corollary \ref{cor:state-transition-maps} to
  \eqref{eqn:w-feedforward} leads to
\begin{equation*}
  w(t) =-U(t)^{-*}U(t_1)^*F^*\ye +
  \int_{t_1}^tU(t)^{-*}U(s)^*C^*\yc \inva s
=: w_h(t) + w_p(t)
\end{equation*}
where $w_h$ describes the feedforward part induced by the nonzero terminal constraint $\ye$
and where $w_p$ describes the inhomogeneity induced by $\yc$.

Then, the expression for $w_h$ follows directly from
\eqref{eqn:utmsuss-toneiss}. With the help of expression \eqref{eqn:utmsuss} for $U(t)^{-*}U(s)^*$, we
calculate
\begin{equation*}% \label{eqn:w-partic}
\begin{split}
  w_p(t) =& \int _{t_1}^t\mxp{(s-t)\APs}\inva s C^*\yc\\
  &- \int _{t_1}^t
  \mxp{(t_1-t)\APs}\tilde S(t_1-t)\mxp{(t_1-s)\AP} W(s-t)\inva s C^*\yc\\
  =& \APms[I-\mxp{(t_1-t)\APs}]C^*\yc \\ 
   &-\mxp{(t_1-t)\APs}\tilde S(t_1-t)\int _{t_1}^t\mxp{(t_1-s)\AP} W(s-t)\inva s
  C^*\yc \\
  =& \APms[I-\mxp{(t_1-t)\APs}]C^*\yc \\ 
   &-\mxp{(t_1-t)\APs}\tilde S(t_1-t)\int
   _{t_1}^t\mxp{(t_1-s)\AP}[W-\mxp{(s-t)\AP}W\mxp{(s-t)\APs}]\inva s C^*\yc  \\
  =& \APms[I-\mxp{(t_1-t)\APs}]C^*\yc \\ 
   &-\mxp{(t_1-t)\APs}\tilde S(t_1-t)\int
   _{t_1}^t\mxp{(t_1-s)\AP}W-\mxp{(t_1-t)\AP}W\mxp{(s-t)\APs}]\inva s
  C^*\yc \\
  =& \APms C^*\yc - \mxp{(t_1-t)\APs}C^*\yc \\ 
   &-\mxp{(t_1-t)\APs}\tilde S(t_1-t)\bigl[\APmo W - \APmo
     \mxp{(t_1-t)\AP}W\bigr ]C^*\yc \\
   &-\mxp{(t_1-t)\APs}\tilde S(t_1-t)\bigl[-\mxp{(t_1-t)\AP}W\APms + \\
   &\quad\quad\quad\quad\quad\quad\quad\quad\quad\quad\quad\quad\mxp{(t_1-t)\AP}W\mxp{(t_1-t)\APs}\APms\bigr ]C^*\yc \\
  =& \APms C^*\yc - \mxp{(t_1-t)\APs}C^*\yc \\ 
   &-\mxp{(t_1-t)\APs}\tilde S(t_1-t)\biggl[ \APmo \bigl[I -
     \mxp{(t_1-t)\AP}\bigr ]W +\\
   &\quad\quad\quad\quad\quad\quad\quad\quad\quad\quad\quad\quad\mxp{(t_1-t)\AP}W\bigl[I -\mxp{(t_1-t)\APs}\bigr]\APms
   \biggr]C^*\yc,
\end{split}
\end{equation*}
This completes the proof.
\end{proof}

% Summing up, we have that 
% \begin{equation}\label{eqn:exp-form-w}
%   w(t) = \APms C^*y_c - \mxp{(t_1-t)\APs}\bigl( I-\tilde S(t_1-t)W(t_1-t)]F^*y_e  - [I+g(t_1, t)]C^*y_c \bigr),
% \end{equation}
% 
% with 
% \begin{equation}
%   g(t_1, t) := 
%    -\tilde S(t_1-t)\bigl[\APmo W - \APmo
%      \mxp{(t_1-t)\AP}W
%    +\mxp{(t_1-t)\AP}W\APms 
%  -\mxp{(t_1-t)\AP}W\mxp{(t_1-t)\APs}\APms\bigr ]
% \end{equation}

Similarly, by means of expression \eqref{eqn:utusmo} for $U(t)U(s)^{-1}$, the
optimal state $x$ as in \eqref{eqn:optix-closedloop} can be expressed as
\begin{equation}\label{eqn:formula-x}
\begin{split}
  x(t) =& U(t)U(0)^{-1}x_0 - \int_0^t U(t)U(s)^{-1}BB^*w(s)\inva s \\
       % =& x_h(t) - \int_0^t [\mxp{(t-s)\AP} - W(t-s)\mxp{(t_1-t)\APs}\tilde
       %   S(t_1-s)\mxp{(t-s)\AP}]BB^*w(s)\inva s \\
       =& x_h(t) - \int_0^t \mxp{(t-s)\AP}BB^*w(s)\inva s + \\
        &\quad \int_0^t W(t-s)\mxp{(t_1-t)\APs}\tilde
        S(t_1-s)\mxp{(t-s)\AP}BB^*w(s)\inva s,
       % =&: x_h(t) - I_1[w] + I_2[w],
\end{split}
\end{equation}
where $x_h$ denotes the solution of the homogeneous problem.
% where $I_1$ and $I_2$ denote the integrals, respectively.
Next, we will use the explicit expression of the feedforward $w$ to explicate
the asymptotic behavior of $x$.

\begin{proposition}\label{lem:formula-x}
  Let Assumption \ref{ass:exist-are-sol-stab} hold and consider
  \begin{equation}\label{eqn:formula-x-xsxhdecay}
    x(t) = x_h(t) + x_s - \mxp{t\AP} \APmo BB^*\APms C^*y_c + g(t,t_1)
    % x(t) = x_s + x_h(t) 
  \end{equation}
  for $x$ as defined in \eqref{eqn:optix-closedloop}, where $x_s$ is the
  steady state solution as defined in Lemma \ref{lem:steady-state-sols},
  where $x_h(t):=U(t)U(0)^{-1}x_0$ with $U$ as in Lemma \ref{lem:fundamental-solution},
  where $\AP$ is defined in \eqref{eqn:app}, and where
  $g(t,t_1)=x(t)-x_h(t)-x_s$ denotes the remainder term.

  If, furthermore, Assumption \ref{ass:szero-invertible} holds, then for every
  fixed $t\in \mathbb R$, the remainder term $g(t,t_1)$ decays to zero
  exponentially as $t_1\to \infty$.
\end{proposition}

% It will turn out that all terms that include $\mxp{(t_1-\tau)\AP}$ will be
% estimated as going to zero as $t_1\to \infty$. Accordingly, the only substantial
% contribution of $w$ to $x$ will be given by $I_1$ and the constant part in

\begin{proof}
  We will use the expression \eqref{eqn:formula-x} of $x$ with the explicit
  formulas for the summands of $w=w_h+w_p$ as stated in Lemma \ref{lem:formula-w}. 
  For estimating the individual terms, we introduce the
  integral operators $I_1[\cdot]$ and $I_2[\cdot]$ via
  \begin{align*}
    I_1[w]:&=\int_0^t \mxp{(t-s)\AP}BB^*w(s)\inva s, 
    \intertext{and }
    I_2[w]:&=\int_0^t W(t-s)\mxp{(t_1-t)\APs}\tilde
    S(t_1-s)\mxp{(t-s)\AP}BB^*w(s)\inva s.
  \end{align*}
  With that and with $w=w_h+w_p$, equation \eqref{eqn:formula-x} becomes
  \begin{equation}
    x(t) = x_h(t) - I_1[w_h]-I_1[w_p] + I_2[w_h] + I_2[w_p].
  \end{equation}

  Next, we will derive estimates or explicit formulas for the contributions of
  the different parts of $w_p$ (as defined in Lemma \ref{lem:formula-w}) to
  $I_1[w]$.
% It will turn out that all terms that include $\mxp{(t_1-\tau)\AP}$ will be
% estimated as going to zero as $t_1\to \infty$. Accordingly, the only substantial
% contribution of $w$ to $x$ will be given by $I_1$ and the 
  For the constant part of $w_p$, we compute
\eqref{eqn:w-partic}:
\begin{equation*}% \label{eqn:optistate-w-const}
  \begin{split}
    -I_1[\APms C^*y_c]&=-\int_0^t \mxp{(t-s)\AP}BB^*\APms C^*y_c \inva s \\
    &= \APmo BB^*\APms C^*y_c -
    \APmo \mxp{t\AP} BB^*\APms C^*y_c % = \xs - \mxp{t\AP} \APmo BB^*\APms C^*y_c,
  \end{split}
\end{equation*}
which, with $x_s=\APmo BB^*\APms C^*y_c$, contributes two of the four summands to
$x(t)$ in formula \eqref{eqn:formula-x-xsxhdecay}.

It remains to show that under the additional Assumption \ref{ass:szero-invertible}, the
remainder part decays to zero as $t_1 \to \infty$. 
In particular, Assumption \ref{ass:szero-invertible} implies that the matrix
valued function $\tilde S$ that appears in the expressions for $w_p$ and $w_h$ is
bounded independent of $t$.
For the uniform boundedness of $W(t)$ (cp.  Remark \ref{rem:uniformW}),
Assumption \ref{ass:exist-are-sol-stab} is sufficient.

For the part $- \mxp{(t_1-t)\APs}C^*\yc$ of $w_p$, we calculate
\begin{equation*}
  \begin{split}
    -I_1[- \mxp{(t_1-t)\APs}C^*\yc] &= \int_0^t
  \mxp{(t-s)\AP}BB^*\mxp{(t_1-s)\APs}C^*\yc \inva s \\
  &=\int_0^t \mxp{(t-s)\AP}BB^*\mxp{(t-s)\APs}\inva s \mxp{(t_1-t)\APs}C^*\yc \\
  &=W(t)\mxp{(t_1-t)\APs}C^*\yc ,
\end{split}
\end{equation*}
which because of the uniform boundedness of $W(t)$ decays to zero exponentially
as $t_1\to \infty$.

For the remaining part of $w_p$ in \eqref{eqn:w-partic}, we note that, by the
stability of \AP,
\begin{equation*}
  c(t_1,t) := \| \Bigl[ \APmo \bigl[I - \mxp{(t_1-t)\AP}\bigr ]W +\mxp{(t_1-t)\AP}W\bigl[I -\mxp{(t_1-t)\APs}\bigr]\APms
  \Bigr]\|
\end{equation*}
is bounded independently of $t_1$ and $t\leq t_1$. With that we can estimate the
remaining contribution of $w_p$ to $I_1[w_p]$ as
\begin{equation*}
  \int_0^t \| \mxp{(t-s)\AP}BB^*\mxp{(t-s)\APs}\|\inva s
  \|\mxp{(t_1-t)\APs}\|\|K(\tilde S)c(t_1, t)\|C^*\yc\|,
\end{equation*}
which, again, is a term that decays as $t_1 \to \infty$.

The contribution by $I_1[w_h]$ (cp. \eqref{eqn:w-homo} and \eqref{eqn:formula-x}) reads
\begin{align*}
  \int_0^t & \mxp{(t-s)\AP}BB^*w_h(s)\inva s \\
=&- \int_0^t \mxp{(t-s)\AP}BB^*\mxp{(t_1-s)\APs}[I-\tilde S(t_1-s)W(t_1-s)]\inva
s F^*\ye \\
=&- \int_0^t \mxp{(t-s)\AP}BB^*\mxp{(t-s)\APs}\mxp{(t_1-t)\APs}[I-\tilde
S(t_1-s)W(t_1-s)]\inva s F^*\ye \\
=&- \int_0^t \mxp{(t-s)\AP}BB^*\mxp{(t-s)\APs}\inva s\mxp{(t_1-t)\APs} F^*\ye \\
 &+ \int_0^t \mxp{(t-s)\AP}BB^*\mxp{(t-s)\APs}\mxp{(t_1-t)\APs}\tilde
 S(t_1-s)W(t_1-s)\inva s F^*\ye.
\end{align*}
With the uniform boundedness of $\tau \to \tilde S(\tau), W(\tau)$, 
% (cp. Lemma \ref{lem:stilde} and Remark \ref{rem:uniformW})
we get the estimate
\begin{equation*}
  \begin{split}
    \|\int_0^t  \mxp{(t-s)\AP}B&B^*w_h(s)\inva s\|  \\
    \leq &\|W(t)\|\|\mxp{(t_1-t)\AP}\|(1+\sup_{0\le \tau \le t}\{\|\tilde
    S(t_1-\tau)W(t_1 - \tau)\|\})\|F^*\ye\| \\
      \leq & \|W\|(1+K(\tilde S)\|W\|\|F^*\ye\|)\mxp{(t_1-t)\lambda} =:
      \const \mxp{(t_1-t)\lambda}.
\end{split}
\end{equation*}

Finally, we note that the integrants of $I_2$ equal the integrants of $I_1$ up
to the factor $W(t-s)\mxp{(t_1-t)\APs}\tilde S(t_1-s)$ which is uniformly
bounded by $K(S)\|W\|\|\mxp{(t_1-t)\APs}\|$. Accordingly, the contribution of
$I_2(w)$ can be estimated by the contributions of $I_1$ times a factor that
includes $\mxp{(t_1-t)\lambda}$ but is independent of $t$ and $t_1$ otherwise. 
\end{proof}

% We collect the above calculations in the following lemma:

% \begin{lemma}\label{lem:formula-x}
%   Let Assumptions \ref{ass:exist-are-sol-stab} and \ref{ass:szero-invertible}
%   hold. Then the solution $x$ to the finite time optimal control problem Problem
%   \ref{prb-fiti-optcont} is given as
% \begin{equation*}
%   x(t) = x_h(t) + x_s - \mxp{t\AP} \APmo C^*y_c + g(t,t_1)
% \end{equation*}
% where $x_h$ solves Problem \ref{prb-fiti-optcont} for $\yc=0$ and $\ye=0$, where
% $x_s$ is the solution to the steady state optimal control problem Problem
% \ref{prb-stst-optcont}, and where $g(t, t_1)$ can be estimated like
% \begin{equation*}
%   \|g(t,t_1)\| \le \const  \mxp{(t_1 - t)\lambda}
% \end{equation*}
% where $\lambda<0$ is the spectral abscissa of \AP.
% \end{lemma}

% \begin{corollary}[of Lemma \ref{lem:formula-x}]%\label{cor:formula-u}
%   With $x$ as in Lemma \ref{lem:formula-x}, the optimal input $u$ for Problem
%   \ref{prb-fiti-optcont} is given as
% \end{corollary}

Proposition \ref{lem:formula-x} directly implies the turnpike property for the
solutions to Problem \ref{prb-fiti-optcont}; cp. Definition \ref{def:turnpike}.

\begin{corollary}
  Under Assumptions \ref{ass:exist-are-sol-stab} and \ref{ass:szero-invertible},
  for the solutions $(x, u)$ to Problem \ref{prb-fiti-optcont} and $(x_s, u_s)$ to
  \ref{prb-stst-optcont} it holds that 
  \begin{align*}
    \|x(t)-x_s\| & \leq \const (\mxp{t\lambda} + \mxp{(t_1 - t)\lambda })
  \intertext{and}
  \|u(t)-u_s\| & \leq \const (\mxp{t\lambda} + \mxp{(t_1 - t)\lambda }),
  \end{align*}
  for a constant $\const >0$ independent of $t_1$ and $\lambda<0$ being the spectral
  abscissa of \AP.
\end{corollary}

\begin{proof}
  With the representation of $x$ as in Proposition \ref{lem:formula-x}, we can
  estimate that
  \begin{equation*}
    \begin{split}
      \|x(t)-x_s\| &\le \|x_h(t)- \mxp{t\AP} \APmo C^*y_c + g(t,t_1)\| \\
                   &\le \|x_h(t)\| + \|\mxp{t\AP} \APmo C^*y_c\|+ \|g(t,t_1)\| \\
                   &\le \const (\mxp{t\lambda} + \mxp{(t_1-t)\lambda}).
  \end{split}
  \end{equation*}
  For the input we recall that by Theorem \ref{thm:locatelli} the optimal input
  is given as $u(t)=-B^*(P(t)x(t)+w(t))$, where $P$ is the solution to the
  differential Riccati equation \eqref{eqn:DRE} and $w$ solves
  \eqref{eqn:w-feedforward}. With $P(t)=\PP+\PD(t)$ and with the formulas
  \eqref{eqn:w-homo} and \eqref{eqn:w-partic} for $w$ we find that
  \begin{equation*}
    u(t) = -B^*\PP x(t) - B^*\PD x(t) - B^*\APs C^* \yc - B^* g_w(t, t_1),
  \end{equation*}
  where $g_w(t, t_1)$ collects all reminder terms of $w$ and which is readily
  estimated by the decay of $\mxp{(t_1-t)\APs}$. With $u_s= -B^*\PP \xs -
  B^*\APs C^* \yc$ (see Lemma \ref{lem:steady-state-sols}), we directly estimate  
  \begin{equation}\label{eqn:turnpike-u}
    \begin{split}
      \|u(t) - u_s\| &= \| -B^*\PP(x(t)-\xs) - B^*\PP \xs \\
                     &\phantom{=\|} - B^*\APs C^* \yc -
    B^*\PD x(t) - B^*g_w(t,t_1) - u_s\| \\
                      & \leq \|B^* \PP (x(t) - \xs) \| + \| B^*\PD(t) x(t) \| +
                      \| B^* g_w(t, t_1)\| \\
                      & \leq \|B^*\PP\| \|x(t) -x_s\| + \|B^*\|\|\PD(t)\| (\|x_s\| +
                      \|x(t) - x_s\|) + \\
                      &\quad \quad +\|B^*\| \|g_w(t,t_1)\|
  \end{split}
  \end{equation}
  from where the turnpike estimate follows directly by the turnpike estimate for
  $x(t)-\xs$, the exponential decay of $g_w(t,t_1)$ with $t_1 \to \infty$, and
  the exponential decay of $\PD(t)$ as $t_1 \to \infty$, see \cite[Thm.
  3]{CalWW94}.
\end{proof}

\section{Linear Quadratic Optimal Control for Descriptor
Systems}\label{sec:lqr-dae}
We now consider optimal control problems with differential algebraic equations
(DAE) of the form
\begin{equation}\label{eqn:dae}
  \daE \dot x(t) = \daA x(t) + \daB u(t), \quad \daE x(0)=\daE x_0.
\end{equation}
as constraints. If the coefficient matrix $\daE$ is not invertible, then the equation
\eqref{eqn:dae} will be made of differential and algebraic equations for $x$,
hence the name DAE.

\begin{problem}[Finite horizon optimal control problem]\label{prb-dae-fiti-optcont}
For coefficients $\daA$, $\daE \in \Rmat nn$, $\daB\in\Rmat nm$, $\daC\in\Rmat kn$, and
$\daF\in
  \Rmat \ell n$$,$ for an initial value $x_0\in \Rvec n$, for target outputs
  $\yc \in \Rvec k$ and $\ye \in \Rvec \ell$ and a terminal time $t_1>0$,
consider
\begin{equation*}
  \frac 12 \int_{0}^{t_1} \|\daC x(s)-y_c\|^2+ \|u(s)\|^2\inva s +\frac 12  \|\daF  x(t_1)-y_e\|^2
  \to \min_u
\end{equation*}
subject to the DAE \eqref{eqn:dae}.
\end{problem}

Problem \ref{prb-dae-fiti-optcont} is a convex problem with affine linear
constraints, which implies that if a candidate solution satisfies first order
necessary optimality conditions, then it is an optimal solution. For $y_c=0$ and $y_e=0$, the \emph{formal} first order necessary
conditions \cite{KunM11a} for Problem \ref{prb-dae-fiti-optcont} read
\begin{equation}\label{eqn:first-order-opti-dae}
  \begin{bmatrix}
    \daE  & 0 \\ 0 & \daE ^*
  \end{bmatrix}
  \frac{d}{dt}
  \begin{bmatrix}
    x \\ p 
  \end{bmatrix}
  =
  \begin{bmatrix}
    \daA  & -\daB \daB^* \\ -\daC ^*\daC  & -\daA ^*
  \end{bmatrix}
  \begin{bmatrix}
    x \\ p 
  \end{bmatrix}, \quad \daE x(0)=\daE x_0, \quad \daE ^*p(t_1) = \daF ^*\daF x(t_1),
\end{equation}
and define the optimal control as $u(t)=-\daB ^*p(t)$.

\begin{remark}\label{rem:formal-real-opticonds}
  The optimality conditions \eqref{eqn:first-order-opti-dae} are called formal,
  because they are formally derived through a variation of the original problem
  formulation.
  However, it is known that the optimal control problem can have a solution
  while the formal optimality conditions do not have a solution \cite{KunM11a}.
  Thus, one should either use an equivalent reformulation of the optimal control
  problem (as proposed in \cite{KunM11a}) or make sure that the formal
  optimality conditions are solvable \cite{Hei16}.
  In this work, we will pursue the second approach.
\end{remark}

One can confirm directly that, if $\daP$ solves the generalized differential
Riccati equation (gDRE)
\begin{equation}\label{eqn:dae-dif-ric}
  -\daE^*\dot\daP = \daA^*\daP + \daP^*\daA - \daP^*\daB\daB^*\daP + \daC^*\daC
  = 0, \quad \daE^* \daP(t_1) = \daF^* \daF,
\end{equation}
then the ansatz $p = \daP x$ decouples the optimality conditions
\eqref{eqn:first-order-opti-dae} and defines a solution.

As in the ODE case, we will consider stabilizing solutions of an associated
generalized algebraic Riccati equation (gARE)
    \begin{equation}\label{eqn:dae-alg-ric}
      \daA^*X + X^*\daA - X^*\daB\daB^*X + \daC^*\daC = 0, \quad
      \daE^*X=X^*\daE.
    \end{equation}

% \section{Notions and Notations for DAEs}

Next, we provide the basic nomenclature and fundamental results for DAEs with
inputs and outputs.

\begin{definition}
  A matrix pair $(\daE, \daA)$ or a matrix pencil $s\daE-\daA$ is called
  \emph{regular}, if there exists an $s\in \mathbb C$ such that $s\daE-\daA$ is
  invertible.
\end{definition}
To introduce the stability concepts we rely on a
direct consequence of a canonical form that can be derived for regular DAEs;
see \cite[Thm. I.2.7]{KunM06}.
\begin{lemma}
  If $(\daE,\daA)$ is regular, then the associated DAE is equivalent to the
  decoupled system
  \begin{equation}\label{eqn:dae-wcf}
  \begin{bmatrix}
    I & 0 \\ 0 & N 
  \end{bmatrix}
  \frac{d}{dt}
  \begin{bmatrix}
    \xsss \\ \xfss
  \end{bmatrix}
  =
  \begin{bmatrix}
    J & 0 \\ 0 & I 
  \end{bmatrix}
  \begin{bmatrix}
    \xsss \\ \xfss
  \end{bmatrix}
  +
  \begin{bmatrix}
    \bsss \\ \bfss
  \end{bmatrix}
  u,
\end{equation}
where $J$ and $N$ are square matrices in \emph{Jordan canonical form} and where
$N$ is \emph{nilpotent} which means that there exists a $\nu \in \mathbb N$ such
that $N^\nu=0$.
\end{lemma}

The part with the state $\xsss$ is called the \emph{slow subsystem} or the
\emph{finite dynamics} and the part of $\xfss$ is called the \emph{fast subsystem}. 

\begin{definition}[Finite dynamics stability]
  Let $(\daE,\daA)$ be regular. 
  \begin{enumerate}
    \item The DAE \eqref{eqn:dae} with coefficients
  $(\daE,\daA)$ is called \emph{finite dynamics stable} if its \emph{slow
  subsystem} is stable, i.e., if all eigenvalues of $J$ in the associated
  canonical form \eqref{eqn:dae-wcf} have negative real part.
\item The triple $(\daE, \daA,  \daB)$ is called \emph{finite dynamics
  stabilizable}, if there exists a feedback matrix $
  \mathcal K$ such that $(\daE,
  \daA-\daB\mathcal K)$ is finite dynamics stable.
  \end{enumerate}
\end{definition}
For equivalent algebraic characterizations and for the duality with
detectability, see \cite[Ch. 3-1.2]{Dai89}.

From the solution formula (\cite[Lem. I.2.8]{KunM06}) for the \emph{fast subsystem} 
\begin{equation}\label{eqn:dae-xfss}
  \xfss(t) = -\sum_{i=1}^{\nu-1} N^i\bfss \frac{d^i}{dt^i} u(t)
\end{equation}
one finds that an initial condition for $\xfss$ that does not equal the
expression of \eqref{eqn:dae-xfss} at $t=0$ generates impulses in the solution;
cp. \cite[Eqn. (2-2.9)]{Dai89}.
If any such impulse can be compensated by an input that is piecewise
$\nu-1$-times differentiable, then the system is called \emph{impulse
controllable}.  
To circumvent the technicalities that come with distributions and to express
\emph{impulse controllability} in terms of the original coefficients
$(\daE,\daA,\daB)$, we will use an equivalent algebraic characterization; see
\cite[Thm. 2-2.3]{Dai89}.
\begin{definition}\label{def:dae-imp-cont}
  Let $(\daE,\daA)$ be regular. The DAE \eqref{eqn:dae} with coefficients
$(\daE,\daA,\daB)$ is called \emph{impulse controllable}, if
\begin{equation*}
  \rank
  \begin{bmatrix}
    \daE & 0 & 0 \\
    \daA & \daE & \daB
  \end{bmatrix}
  = n + \rank \daE.
\end{equation*}
\end{definition}

A matrix pencil $s\daE-\daA$ is called \emph{impulse-free} if no impulses occur in the
DAE solution regardless of the initial value. This means that it is trivially
\emph{impulse controllable}. In line with Definition
\ref{def:dae-imp-cont}, this can be characterized as follows
\begin{definition}\label{def:dae-imp-free}
  Let $(\daE,\daA)$ be regular. The DAE \eqref{eqn:dae} with coefficients
  $(\daE,\daA)$ is called \emph{impulse-free}, if
\begin{equation*}
  \rank
  \begin{bmatrix}
    \daE & 0 \\
    \daA & \daE
  \end{bmatrix}
  = n + \rank \daE.
\end{equation*}
\end{definition}

\begin{remark}
  As for standard systems, the notions of \emph{finite time detectability} and
  \emph{impulse observability} of a DAE with output matrix $\daC$ can be defined
  by duality, i.e., via the \emph{finite time stability} and \emph{impulse
  controllability} of $(\daE^*, \daA^*, \daC^*)$; see \cite[Thm. 2-4.1]{Dai89}. 
\end{remark}

For semi-explicit \emph{impulse-free} systems, \emph{finite-dynamics stability} can be
characterized as follows:
\begin{lemma}\label{lem:findyn-stab}
  Let 
  \begin{equation*}
    \bigl(
      \begin{bmatrix}
        I & 0 \\ 0 & 0
      \end{bmatrix},
      \begin{bmatrix}
        \ao  & \aot  \\ \ato  & \at 
      \end{bmatrix}
    \bigr)
  \end{equation*}
  be a regular impulse-free matrix pair. Then it is \emph{finite dynamics stable} if,
  and only if, $\ao -\aot \at ^{-1}\aot $ is stable.
\end{lemma}

\begin{proof}
  By the \emph{impulse-free} assumption it follows that $\at $ is invertible. From the
  equality
  \begin{equation*}
    \begin{split}
    \begin{bmatrix}
      I & -\aot \at ^{-1} \\ 0 & I 
    \end{bmatrix}
    \begin{bmatrix}
      -sI+\ao  & \aot  \\ \ato  & \at 
    \end{bmatrix}
    &\begin{bmatrix}
      I & 0 \\ -\at ^{-1}\ato  & I
    \end{bmatrix}
    = \\
    &\quad \begin{bmatrix}
      -sI+\ao -\aot \at ^{-1}\ato  & 0 \\ 0 & I
    \end{bmatrix},
  \end{split}
  \end{equation*}
  we find that
  \begin{equation*}
    \bigl(
      \begin{bmatrix}
        I & 0 \\ 0 & 0
      \end{bmatrix},
      \begin{bmatrix}
        \ao  & \aot  \\ \ato  & \at 
      \end{bmatrix}
    \bigr)
    \quad \text{is equivalent to}\quad
    \bigl(
      \begin{bmatrix}
        I & 0 \\ 0 & 0
      \end{bmatrix},
      \begin{bmatrix}
        \ao  - \aot  \at ^{-1}\ato  & 0 \\ 0 & I
      \end{bmatrix}
    \bigr)
  \end{equation*}
  from where we deduce that the \emph{slow subsystem} is stable if, and only if,
  the matrix $\ao  - \aot  \at ^{-1}\ato $ is stable.
\end{proof}

Next, we state the underlying assumptions for our analysis and some immediate
consequences.

\begin{assumption}\label{ass:exist-gare-sol-stab}
The coefficients $\daE$, $\daA$, $\daB$, and $\daC$ in Problem \ref{prb-dae-fiti-optcont}
and \eqref{eqn:dae} are such
\begin{enumerate}
  \item that the pair $(\daE,\daA)$ is regular and
  \item that the gARE \eqref{eqn:dae-alg-ric}
    has a stabilizing solution \daPP, i.e., the pair 
    \begin{equation*}
      (\daE, \daAP) := (\daE, \daA-\daB\daB^*\daPP)
    \end{equation*}
    is regular, \emph{impulse-free}, and \emph{finite dynamics stable}.
\end{enumerate}
\end{assumption}

Assumption \ref{ass:exist-gare-sol-stab} implies that the system $(\daE,\daA,\daB)$ is \emph{finite
dynamics stabilizable} and \emph{impulse controllable}. We will show below, that
\emph{impulse observability} of $(\daE,\daA,\daB)$ is a necessary condition for
existence of a stabilizing Riccati solution $\daPP$. As for the ODE case,
\emph{(finite time) detectability} is not necessary for the existence of
\daPP; cp. Remark \ref{rem:are-detectability}.

For existence of solutions to the necessary optimality conditions
\eqref{eqn:first-order-opti-dae} in general, we make the following assumption 
\begin{assumption}\label{ass:rangeF-rangeE}
  The matrices $\daE$ and $\daF$ in Problem \ref{prb-dae-fiti-optcont} are compatible
  in the sense that
  \begin{equation*}
    \mathop{range} \daF^* \subset \mathop{range} \daE^* .
  \end{equation*}
\end{assumption}

Looking at the terminal condition $\daE^*p (t_1)=\daF^*\daF x(t_1)$, one can find
that Assumption \ref{ass:rangeF-rangeE} is necessary for existence of solutions
to the optimality conditions \eqref{eqn:first-order-opti-dae}. 
Also it is an implicit assumption made in \cite[cp. Equation (1)]{BenL87} and
the base for more general results (see, e.g., \cite[Thm. 13]{KunM08}).
Still it is not a necessary condition for existence of optimal solutions; cp.
Remark \ref{rem:formal-real-opticonds}.

In order to simplify the formulas, we further assume that $\daE$ has a
\emph{semi-explicit} structure, which means that the differentiated variables
appear explicitly:

\begin{assumption}\label{ass:E-semexp}
  The matrix $\daE \in \mathbb R^{n,n}$ in \eqref{eqn:dae} is of the form 
  \begin{equation}\label{eqn:semexp-E}
    \daE = 
    \begin{bmatrix}
      I & 0 \\
      0 & 0
    \end{bmatrix}
  \end{equation}
  where $I \in \mathbb R^{d,d}$ is the identity matrix.
\end{assumption}

\begin{remark}
In theory, this assumption is not restrictive since a regular transformation of
the system can always provide such a form of $\daE$. In practice, to actually
solve the Riccati equations, this semi-explicit realization of the state
equations might be helpful. While in the general and, in particular, large scale
case, such a transformation may destroy sparsity and introduce
systematic errors to the model, we want to remark that for many
applications the invertible part of $\mathcal E$ can be inferred from the
structure so that it can be transformed into the semi-explicit form in a
computationally feasible way.
\end{remark}

\begin{remark}
  Assumption \ref{ass:E-semexp} and the symmetry constraint $\daE^*\daPP =
  \daPP^*\daE$
  imply that $\daPP$ is a block lower-triangular matrix, i.e.
  \begin{equation}\label{eqn:dae-pp-partition}
    \daPP =
    \begin{bmatrix}
      \daPPii & 0 \\ P_{+;21} & \daPPtt
    \end{bmatrix}
  \end{equation}
  with \daPPii~being symmetric, i.e. $\daPPii^*=\daPPii$. Moreover, Assumptions
  \ref{ass:E-semexp} and \ref{ass:rangeF-rangeE} together imply that  

  \begin{equation}\label{dae:fsf}
    \daF^*\daF = \begin{bmatrix} \daSi & 0 \\ 0 & 0 \end{bmatrix}.
  \end{equation}
\end{remark}

\begin{remark}
For $\daE$ in semi-explicit form, several concepts can be made more explicit. Let 
\begin{equation}\label{eqn:partition-daABC}
  \daA = 
  \begin{bmatrix}
    \ao & \aot \\ \ato & \at
  \end{bmatrix}, \quad
  \daB = 
  \begin{bmatrix}
    \bo \\ \bt
  \end{bmatrix}, \quad\text{and}\quad
  \daC =
  \begin{bmatrix}
    \co & \ct
  \end{bmatrix},
\end{equation}
then $(\daE, \daA, \daB, \daC)$ are impulse controllable or impulse observable
if, and only if, 
\begin{equation}\label{eqn:imp-cont-obs-semexp}
  \begin{bmatrix} \at & \bt \end{bmatrix}\quad\text{or}\quad
  \begin{bmatrix} \at^* & \ct^* \end{bmatrix}
\end{equation}
have full rank, respectively. And $(\daE, \daA)$ is impulse-free if, and only if,
$\at$ is invertible.
\end{remark}

The following lemma relates the associated Hamiltonian matrix pencil to the
existence of stabilizing solutions of the generalized algebraic Riccati equation
\eqref{eqn:dae-alg-ric}.
Although it is the direct extension of the standard state space result, it has
not been stated explicitly so far.

\begin{lemma}
  Let $(\daE, \daA)$ be regular.
  The gARE \eqref{eqn:dae-alg-ric} has a stabilizing solution if, and only if,
  $(\daE,\daA,\daB)$ is finite dynamics stabilizable and the matrix pencil
  \begin{equation*}
\mathcal H(s) =
  \begin{bmatrix}
     -s\daE+\daA & -\daB\daB^* \\ -\daC^*\daC & -s\daE^*-\daA^*
  \end{bmatrix}
\end{equation*}
is regular, impulse-free, and has no finite eigenvalues on the imaginary axis.
\end{lemma}

\begin{proof}
The necessity is stated and proved in the first lines of the proof of \cite[Lem.
1]{KawTK99}. The sufficiency follows by the arguments of \cite[Sec. 3]{MRS11} as
follows. With $\daE$ in semi-explicit form, the absence of impulses in $\mathcal
H(s)$ implies that 
\begin{equation*}
  \begin{bmatrix}
    \at & -B_2 B_2 ^* \\
    -C_2^*C_2 & \at^*
  \end{bmatrix}
\end{equation*}
is invertible so that $(\daE,\daA, \daB, \daC)$ must be impulse controllable and
impulse observable; cp. \eqref{eqn:imp-cont-obs-semexp}. 
Thus, the condition of \cite[Thm. 3.2]{MRS11} are fulfilled up to the finite
dynamics observability of $(\daE,\daA,\daC)$. 
Still, one can apply \cite[Lem. 3.10]{MRS11} since the
needed invertibility is guaranteed by $\mathcal H(s)$ having no finite modes on
the imaginary axis which implies that $ \begin{bmatrix}\daC ^* & \daA ^*\end{bmatrix}$
has full rank (see \cite[Thm. 4]{Mar71}) as it is needed in the proof of
\cite[Lem. 3.8]{MRS11}.
Thus, existence of the relevant stabilizing solution, which is denoted by $Y$ in
\cite{MRS11}, follows as laid out in \cite[Sec.  3.2]{MRS11}.
\end{proof}

\section{Existence and Asymptotic Behavior of Structured Solutions to the
Differential Riccati Solution}\label{sec:exist-decay-gDRE}

In this section, we establish the existence of a particularly structured solution to the generalized
differential Riccati equation gDRE \eqref{eqn:dae-dif-ric} under the assumption
that the associated algebraic Riccati equation gARE \eqref{eqn:dae-alg-ric} has a
stabilizing solution.

We start with adding another assumption that will be shown to be a sufficient
criterion for existence of solutions to the (gDRE) \eqref{eqn:dae-dif-ric} and
that we will justify by considering its necessity for particular cases.

\begin{assumption}\label{ass:dae-zdg-cond}
  Let Assumptions \ref{ass:exist-gare-sol-stab} and \ref{ass:E-semexp} hold and
  let $(\daE,\daA,\daC)$ and the stabilizing solution \daP~to the gARE
  \eqref{eqn:dae-alg-ric} be partitioned as in \eqref{eqn:dae-pp-partition} and
  \eqref{eqn:partition-daABC}. 
  Then with the
  \daPPtt~block of \daP~ and with $K_2:=A_{22}^* - \daPPtt B_2B_2^*$ being
  regular (cp. \cite[Lem. 1]{KawTK99}):
  \begin{equation}\label{eqn:dae-zdg-cond}
    \begin{split}
    \tilde Q:= C_1^*C_1 
    -&(A_{21}^*P_2+C_1^*C_2)K_2^{-*}A_{21}-A_{21}^*K_2^{-1}(C_2^*C_1+P_2^*A_{21})
    \\ &-(C_1^*C_2+A_{21}^*\daPPtt)K_2^{-*}B_2B_2^*K_2^{-1}(C_2^*C_1+\daPPtt^*A_{21})
    \geq 0.
  \end{split}
  \end{equation}
  % is positive semidefinite.
\end{assumption}

\begin{remark}
  We note that the \daPPtt~block of \daPP~is not uniquely defined, it only has
  to fulfill the quadratic equation
  \begin{equation}\label{dae:ptt-riccati}
    \at^* \daPt + \daPt^*\at - \daPt^*B_2B_2^*\daPt + C_2^*C_2 = 0
  \end{equation}
  and the condition that $\at^* - \daPt^*B_2B_2$ is regular. See 
  \cite[Lem. 1]{KawTK99} or the solution representation provided in 
  \cite[Lem. 3.10]{MRS11} for the semi-explicit \daE. Also, cp. the nonuniqueness of
  the feedback law provided in \cite[Eqn. (3.43)]{KatM92}.
\end{remark}

Before we state global existence of solutions to the generalized differential Riccati
equations, we show that Assumption \ref{ass:dae-zdg-cond} generalizes the
general assumption of a positive definite cost functional for problems that
are impulse-free. 

For that we consider a system that is impulse-free and that, without loss of
generality, can be assumed in the form of

\begin{equation}\label{eqn:dae-example-indone}
  \daE = 
  \begin{bmatrix}
    I & 0 \\ 0 & 0
  \end{bmatrix}, \quad
  \daA = 
  \begin{bmatrix}
    \ao & 0 \\ 0 & -I
  \end{bmatrix}, \quad
  \daB = 
  \begin{bmatrix}
    \bo \\ \bt
  \end{bmatrix}, \quad\text{and}\quad
  \daC =
  \begin{bmatrix}
    \co & \ct
  \end{bmatrix}.
\end{equation}

Like for every system with $(\daE, \daA)$ impulse-free, Problem \ref{prb-dae-fiti-optcont} with system
\eqref{eqn:dae-example-indone} as constraint is equivalent to a standard LQR
problem. In this case, if the state $x=(x_1, x_2)$ is partioned accordingly, one
can express $x_2$ as $x_2 = B_2u$ and find that the cost functional is positive
definite if, and only if, 
\begin{equation*}
  C_1^*C_1 - C_1^*C_2B_2B_2^*C_2^*C_1 \geq 0 \quad\text{or}\quad I - C_2B_2B_2^*C_2^*
  \geq 0 ,
\end{equation*}
which is equivalent to the largest singular vector of $C_2B_2$ being less than one, i.e.
$\bar \sigma (C_2B_2 ) \leq 1$. Thus, for systems in the form
\eqref{eqn:dae-example-indone}, the condition $\bar \sigma (C_2B_2 ) \leq 1$ is
necessary to be in line with the standard theory (cp. \cite[Eqn. (14.2)]{ZhoDG96} or
\cite[Rem. 3.4]{Loc01}).

\begin{lemma}\label{lem:indone-zdg}
  Consider a system $(\daE, \daA, \daB, \daC)$ in the form of
  \eqref{eqn:dae-example-indone}. There exists
  a $\daPt$ that solves \eqref{dae:ptt-riccati}, such that
  $\at^*-\daPt^*B_2B_2^*$ is regular and such that \eqref{eqn:dae-zdg-cond}
  holds if, and only if, $\bar \sigma (C_2B_2 ) \leq 1$.
\end{lemma}

\begin{proof}
  \emph{Sufficieny}: For a system in the form of \eqref{eqn:dae-example-indone}, i.e. $\at = -I$,
  by standard theory, the Riccati equation \eqref{dae:ptt-riccati} has a
  symmetric positive definite or positive semi-definite solution \daPt~such that
  $-I-\daPt ^* B_2B_2^*$ is stable and, thus, invertible. With $\ato=0$, 
  condition \eqref{eqn:dae-zdg-cond} reads 
  \begin{equation*}
    C_1^*C_1 -
    C_1^*C_2(-I-B_2B_2^*\daPt)^{-1}B_2B_2^*(-I-\daPt^*B_2B_2^*)^{-1}C_2^*C_1
    \geq 0
  \end{equation*}
  or, with the identity $(I+B_2B_2^*\daPt)^{-1}B_2=B_2(I+B_2^*\daPt B_2)^{-1}$,
  \begin{equation}\label{eqn:dae-zdg-cond-indone}
    I -
    C_2B_2(I+B_2^*\daPt B_2)^{-1}(I+B_2^*\daPt^*B_2)^{-1}B_2^*C_2^*
    \geq 0.
  \end{equation}
  With $\daPt\geq 0$, all eigenvalues of $I+B_2^*\daPt B_2$ are
  larger than one and, accordingly, all eigenvalues of $(I+B_2^*\daPt B_2)^{-1}$ are
  smaller than one. Since for symmetric positive definite matrices, the
  eigenvalues coincide with the singular values and since
  $\bar \sigma$ has all the properties of a norm, we can estimate
  \begin{equation*}
    \begin{split}
    \bar \sigma (C_2(I+B_2B_2^*\daPt)^{-1}B_2) 
    & =
    \bar \sigma (C_2B_2(I+B_2B_2^*\daPt)^{-1}) \\
    & \leq \bar \sigma(C_2B_2) \bar \sigma ((I+B_2B_2^*\daPt)^{-1})
    \leq \bar \sigma(C_2B_2) \leq 1
  \end{split}
  \end{equation*}
  from where we conclude \eqref{eqn:dae-zdg-cond-indone}.

  \emph{Necessity}: The multiplication of the corresponding Riccati equation by
  $B_2^*$ and $B_2$ from the left and the right, respectively, and the completion of the square
  yield
  \begin{equation*}
    (I - B_2^*\daPt B_2)^*(I - B_2^*\daPt B_2) = I - B_2^*C_2^*C_2 B_2
  \end{equation*}
  which, by the positive definiteness of the left hand side,  can only hold if $\bar \sigma(C_2 B_2)\leq 1$.
\end{proof}

By Lemma \ref{lem:indone-zdg}, for impulse-free systems, validity of
Assumption \ref{ass:dae-zdg-cond} is implied by the assumption that the
underlying cost functional is positive definite.

\begin{theorem}\label{thm:dae-existence-dre-sol}
  Consider the DAE \eqref{eqn:dae} and Problem
  \ref{prb-dae-fiti-optcont}. Assume that \daE~is semi-explicit and that \daF~is
  compatible (Assumptions \ref{ass:E-semexp} and \ref{ass:rangeF-rangeE}).
  Assume that $(\daE, \daA)$ is regular and that a stabilizing solution
  to the gARE exists (Assumption \ref{ass:exist-gare-sol-stab}). Let Assumption
  \ref{ass:dae-zdg-cond} hold. Then the gDRE \eqref{eqn:dae-dif-ric} has a
  solution \daP~for $t\leq t_1$. 
\end{theorem}

\begin{proof}
  Since $\daE^*\dot \daP(t)$ is symmetric and $\daE^*\daP(t_1)=\daF^*\daF$ is
  symmetric, it holds that $\daE^*\daP(t)$ is symmetric or, due to the
  semi-explicit form of \daE, that
  \begin{equation*}
    \daP = 
    \begin{bmatrix}
      \daPo & 0 \\ \daPto & \daPt
    \end{bmatrix}
  \end{equation*}
  with $\daPo(t)$ being symmetric. With this block triangular structure and the
  partition of the coefficients $(\daA, \daB, \daC)$ as in
  \eqref{eqn:partition-daABC}, the gDRE \eqref{eqn:dae-dif-ric} can be written in
  terms of the following four coupled matrix valued equations:
  \begin{subequations}
  \begin{align}
    \label{eqn:dae-dre-oo}
      -\dot \daPo &= \ao^*\daPo + \ato^*\daPto + \daPo \ao + \daPto^*\ato \notag \\
                 & \mspace{50mu}- \daPo B_1B_1^*\daPo - \daPto^* B_2B_1^*\daPo - \daPo B_1B_2^*\daPto -
                  \daPto^*B_2B_2^*\daPto + C_1^*C_1, \notag \\
                 & \daPo(t_1) = S, \\
      \label{eqn:dae-dre-to}
      0&=\aot^*\daPo +\at^*\daPto + \daPt ^*\ato - \daPt^* B_2B_1^*\daPo - \daPt ^*B_2B_2^*\daPto +
      C_2^*C_1,\\
      \label{eqn:dae-dre-ot}
      0&=\daPo\aot + \daPto^*\at + \ato^*\daPt - \daPo^*B_1B_2^*\daPt -
      \daPto^* B_2B_2^*\daPt + C_1^*C_2, \\
      \label{eqn:dae-dre-tt}
      0&=\at^*\daPt + \daPt ^*\at - \daPt ^*B_2B_2^*\daPt + C_2^*C_2.
  \end{align}
\end{subequations}
  Note that \eqref{eqn:dae-dre-ot} is the transpose of and, thus, equivalent
  to \eqref{eqn:dae-dre-to}.

  Since \eqref{eqn:dae-dre-tt} does not differ from the left-lower block of the
  gARE \eqref{eqn:dae-alg-ric}, Assumption \ref{ass:exist-gare-sol-stab} implies
  the existence of a (constant) matrix (function) $\daPt$ that solves
  \eqref{eqn:dae-dre-tt} such that 
  \begin{equation*}
    K_2 := \at^*-\daPt^*B_2B_2^*
  \end{equation*}
  is invertible. Accordingly, the matrix valued function \daPto~is defined by
  virtue of \eqref{eqn:dae-dre-ot} or \eqref{eqn:dae-dre-to} as
  \begin{equation}\label{eqn:dae-dre-pto-resolved}
    \daPto(t) = -K_2^{-1}(\aot^*\daPo(t) + \daPt^* \ato - \daPt^* B_2B_1^*\daPo (t)
    + C_2^*C_1)
  \end{equation}
  and existence of \daP~relies on the existence of a Riccati solution to
  \eqref{eqn:dae-dre-oo} which, with \daPto~expressed in terms of a linear relation
  with \daPo~and a constant term as in \eqref{eqn:dae-dre-pto-resolved}, reads
  \begin{equation}\label{eqn:dae-dre-poo-resolved}
    -\dot \daPo = \tilde A ^* \daPo + \daPo \tilde A - \daPo \tilde R \daPo +
    \tilde Q, \quad \daPo(t_1) = S,
  \end{equation}
  where $\tilde Q$ is as in Assumption \ref{ass:dae-zdg-cond} and where
  \begin{align*}
    \tilde A :&= \ao -(\ato-B_1B_2^*\daPt ) K_2^{-*} \ato +
    B_1B_2^*K_2^{-1}(\daPt \ato + C_2^*C_1)\\
    \tilde R:&= 
    \begin{bmatrix} B_1 &-(B_1B_2^*\daPt - \aot^*)K_2^{-*}B_2 \end{bmatrix}
    \begin{bmatrix} B_1 &-(B_1B_2^*\daPt - \aot^*)K_2^{-*}B_2 \end{bmatrix}^* .
    % \begin{bmatrix} B_1^* \\ -B_2^*K_2^{-1}(\daPt^*B_2B_1^* - \aot) \end{bmatrix} \\
    % B_1B_1^* - B_1B_2^*K_2^{-1}(\aot^*-\daPt^*B_2B_1^*) - (\ato-B_1B_2^*\daPt )K_2^{-*}B_2B_1^* \notag \\ &\mspace{50mu} + (\ato-B_1B_2^*\daPt)K_2^{-*}B_2B_2^*K_2^{-1}(\aot^*-\daPt^*B_2B_1^*),
    % \tilde Q :&= C_1^*C_1 - (C_1^*C_2+A_{21}^*\daPt)K_2^{-*}B_2B_2^*K_2^{-1}(C_2^*C_1+\daPt^*A_{21}).
  \end{align*}
  With $\tilde R \geq 0$, for arbitrary $S\geq 0$, global existence of the
  unique solution to \eqref{eqn:dae-dre-poo-resolved} is ensured (cp.
  \cite[Thm. 4.1.6]{AboFIJ03}), if also $\tilde Q$ is positive semi-definite which it is by
  Assumption \ref{ass:dae-zdg-cond} and with the choice $\daPt=\daPPtt$.
\end{proof}

For this solution to the gDRE, in analogy to the standard ODE case
\cite{CalWW94}, we can show the exponential decay towards the gARE solution.

To prepare the arguments, we consider the associated Hamiltonian boundary value
problem
\begin{equation}\label{eqn:dae-hami-sys}
  \begin{bmatrix}
    \daE  & 0 \\ 0 & \daE ^*
  \end{bmatrix}
  \frac{d}{dt}
  \begin{bmatrix}
    V_1 \\ V_2
  \end{bmatrix}
  =
  \begin{bmatrix}
    \daA & -\daB\daB^* \\ -\daC^*\daC & -\daA^*
  \end{bmatrix}
  \begin{bmatrix}
    V_1 \\ V_2
  \end{bmatrix}, 
  \quad \daE  V_1(t_1) = \begin{bmatrix} I \\ 0 \end{bmatrix}, 
  \quad \daE  V_2(t_1) = \begin{bmatrix} S \\ 0 \end{bmatrix},
\end{equation}
with $V_1(t),$ $V_2(t) \in \mathbb R^{n,d}$, where $d$ is the size of  identity
block in $\daE $ (cp. Assumption \ref{ass:E-semexp}) and where the initial
conditions already anticipate the semi-explicit form of \daE.

With \daPP~solving the gARE \eqref{eqn:dae-alg-ric} and being partitioned as in \eqref{eqn:dae-pp-partition}, we can make use of the transformation
\begin{equation}\label{eqn:dae-pptransf-pencil}
  \begin{bmatrix}
    I & 0 \\ -\daPPs & I 
  \end{bmatrix}
  \begin{bmatrix}
    -s\daE+\daA & -\daB\daB^* \\ -\daC^*\daC & -s\daE^*-\daA^*
  \end{bmatrix}
  \begin{bmatrix}
    I & 0 \\ \daPP & I 
  \end{bmatrix}
  =
  \begin{bmatrix}
    -s\daE+\AP & -\daB\daB^* \\ 0 & -s\daE^*-\APs
  \end{bmatrix}
\end{equation}
to write \eqref{eqn:dae-hami-sys} as 
\begin{equation}\label{eqn:dae-aplus-hami-sys}
  \begin{bmatrix}
    \daE & 0 \\ 0 & \daE^*
  \end{bmatrix}
  \frac{d}{dt}
  \begin{bmatrix}
    V_1 \\ \tilde V_2
  \end{bmatrix}
  =
  \begin{bmatrix}
    \daAP & -\daB\daB^* \\ 0 & -\daAP^*
  \end{bmatrix}
  \begin{bmatrix}
    V_1 \\ \tilde V_2
  \end{bmatrix},
  \quad \daE V_1(t_1) = \begin{bmatrix} I \\ 0 \end{bmatrix}, 
  \quad \daE \tilde V_2(t_1) = \begin{bmatrix} \daSi-\daPPii \\ 0 \end{bmatrix},
\end{equation}
with 
\begin{equation*}
  \begin{bmatrix}
    V_1 \\ \tilde V_2
  \end{bmatrix}
  =
  \begin{bmatrix}
    I & 0 \\ -\daPP & I 
  \end{bmatrix}
  \begin{bmatrix}
    V_1 \\ V_2
  \end{bmatrix}.
\end{equation*}
In line with the semi-explicit structure of $\daE$, we further differentiate 
\begin{equation}\label{eqn:partition-ABV}
  \daAP = 
  \begin{bmatrix}
    \apo & \apot \\ \apto & \apt
  \end{bmatrix}, \quad
  \daB\daB^* = 
  \begin{bmatrix}
    \boo  & \btos  \\ \bto  & \btt 
  \end{bmatrix}, \quad
  \begin{bmatrix}
    V_1 \\ \tilde V_2
  \end{bmatrix}=
  \begin{bmatrix}
  V_{11} \\ V_{12} \\ \tilde V_{21} \\ \tilde V_{22}
  \end{bmatrix}
  .
\end{equation}
With this partitioning and with the swap of the third and second line and
column, respectively, the system reads
\begin{equation}\label{eqn:dae-hami-aplus-resort}
  \begin{split}
  \begin{bmatrix}
    I & 0 & 0 & 0 \\
    0 & I & 0 & 0 \\
    0 & 0 & 0 & 0 \\
    0 & 0 & 0 & 0
  \end{bmatrix}
  \frac{d}{dt}
  \begin{bmatrix}
  V_{11} \\  \tilde V_{21} \\ V_{12} \\\tilde V_{22}
  \end{bmatrix}
  =
  \begin{bmatrix}
    \apo  &  -\boo   & \apot  & -\btos   \\
    0   & -\apo ^* & 0   & -\apto ^* \\
    \apto  & -\bto   & \apt  &  -\btt   \\
    0   & -\apot ^* & 0   & -\apt ^*
  \end{bmatrix}
  \begin{bmatrix}
  V_{11} \\  \tilde V_{21} \\ V_{12} \\ \tilde V_{22}
  \end{bmatrix}, \\
  \quad V_{11}(t_1)= I,
  \quad \tilde V_{21}(t_1)= S-\daPPii.
\end{split}
\end{equation}

Since \AP~is impulse-free, the block $\apt $ is invertible so that with the relation
\begin{equation}\label{eqn:resolve-v12-v22}
  \begin{split}
    \begin{bmatrix}
      V_{12} \\ \tilde V_{22}
    \end{bmatrix}&=-
    \begin{bmatrix}
      \apt  &  -\btt  \\
      0   & -\apt ^*
    \end{bmatrix}^{-1}
    \begin{bmatrix}
      \apto  & -\bto \\
      0   & -\apot ^* 
    \end{bmatrix}
    \begin{bmatrix}
      V_{11} \\ \tilde V_{21}
    \end{bmatrix}\\&=-
    \begin{bmatrix}
      \apt ^{-1} &  -\apt ^{-1}\btt \apt ^{-*}\\
      0   & -\apt ^{-*}
    \end{bmatrix}
    \begin{bmatrix}
      \apto  & -\bto \\
      0   & -\apot ^* 
    \end{bmatrix}
    \begin{bmatrix}
      V_{11} \\ \tilde V_{21}
    \end{bmatrix}\\&=-
    \begin{bmatrix}
      \apt ^{-1} \apto  & -\apt ^{-1} \bto +\apt ^{-1}\btt \apt ^{-*}\apot ^*\\
      0   & \apt ^{-*}\apot ^*
    \end{bmatrix}
    \begin{bmatrix}
      V_{11} \\ \tilde V_{21}
    \end{bmatrix}
  \end{split}
\end{equation}
we get the reduced system
\begin{equation}\label{eqn:dae-reduced-cl-hamiltonian}
  \begin{split}
  &\begin{bmatrix}
    I & 0  \\
    0 & I 
  \end{bmatrix}
  \frac{d}{dt}
  \begin{bmatrix}
  V_{11} \\  \tilde V_{21}
  \end{bmatrix}
  = \\
  & \begin{bmatrix}
    \apo  - \apot  \apt ^{-1} \apto  &  -[B_1 -\apot \apt ^{-1}B_2 ][B_1 -\apot \apt ^{-1}B_2 ]^*  \\
    0   & -(\apo ^* - \apto ^* \apt ^{-*}\apot ^*)
  \end{bmatrix}
  \begin{bmatrix}
  V_{11} \\  \tilde V_{21} 
  \end{bmatrix}, \\
  & \quad \quad \quad \quad \quad \quad \quad \quad \quad \quad \quad\quad V_{11}(t_1)= I,  \tilde V_{21}(t_1)= S-\daPPii.
\end{split}
\end{equation}
as it can be derived by means of the explicit representation of the Schur complement
\begin{equation*}
  \begin{split}
    S := 
    \begin{bmatrix}
      \apot  & -\btos \\
      0   & -\apto ^* 
    \end{bmatrix}
    \begin{bmatrix}
      \apt  &  -\btt  \\
      0   & -\apt ^*
    \end{bmatrix}^{-1}
    \begin{bmatrix}
      \apto  & -\bto \\
      0   & -\apot ^* 
    \end{bmatrix}=\\
    \begin{bmatrix}
      \apot  \apt ^{-1} \apto  & \apot \apt ^{-1}\btt \apt ^{-*}\apot ^*-\apot \apt ^{-1} \bto -\btos \apt ^{-*}\apot ^*\\
      0   & -\apto ^* \apt ^{-*}\apot ^*
    \end{bmatrix}.
  \end{split}
\end{equation*}

In what follows, we will use the abbreviations 
\begin{equation*}%\label{eqn:dae-def-ba-bb}
  \bA:= \apo  - \apot  \apt ^{-1} \apto \quad \text{and}\quad \bB:=[B_1 -\apot
  \apt ^{-1}B_2 ].
\end{equation*}

% \begin{lem}\label{lem:dae-ba-stab}
%   If the matrix pair
%   \begin{equation}
%     (
%     \begin{bmatrix} I & 0 \\ 0 & 0 \end{bmatrix},
%     \begin{bmatrix}
%       \apo & \apot \\ \apto & \apt
%     \end{bmatrix}
%     )
%   \end{equation}
%   is impulse-free and finite-dynamics stable, then $\bA$ as defined in
%   \eqref{eqn:dae-def-ba-bb} is stable.
% \end{lem}

% \begin{proof}
%   This result is standard and follows from a direct computation of the
%   associated canonical form \eqref{eqn:dae-wcf}.
% \end{proof}

\begin{theorem}\label{thm:v11-gDRE-PD}
  Consider the gDRE \eqref{eqn:dae-dif-ric} with $\daE$ semi-explicit as in
  Assumptions \ref{ass:E-semexp} and let \daF~be compatible as in Assumption
  \ref{ass:rangeF-rangeE}. 
  Let the coefficients $(\daA,\daB,\daC)$ be partitioned as in
  \eqref{eqn:partition-daABC} in accordance with $\daE$.  
  Let Assumption \ref{ass:exist-gare-sol-stab} hold, let $\daPP$ be the
  stabilizing solution to the gARE \eqref{eqn:dae-alg-ric}, and let
  $\daAP:=\daA-\daB\daB^*\daPP$ be partitioned as in \eqref{eqn:partition-ABV}.
  Let Assumption \ref{ass:dae-zdg-cond} hold and let \daP~be the solution to the
  gDRE \eqref{eqn:dae-dif-ric}.
  Then system \eqref{eqn:dae-aplus-hami-sys} has a unique solution
  \begin{equation*}
    \begin{bmatrix}
      V_1 \\ \tilde V_2
    \end{bmatrix}=
    \begin{bmatrix}
      V_{11} \\ V_{12} \\ \tilde V_{21} \\ \tilde V_{22}
    \end{bmatrix}
  \end{equation*}
  with $V_{11}(t)$ being invertible for $t\leq t_1$ and such that for $\daPD(t) := \daP(t) -\daPP$ it
  holds that
  \begin{equation}\label{eqn:form-of-PD}
    \begin{split}
      \daPD(t)& =
    \begin{bmatrix}
      \daPDii(t) & 0 \\
      \daPDit(t) & 0
    \end{bmatrix}
    =
    \begin{bmatrix}
      \daPDii(t) & 0 \\
      -\apt ^{-*}\apot ^*\daPDii(t) & 0
    \end{bmatrix} \\
    & =
    \begin{bmatrix}
      \tilde V(t)_{21}V(t)_{11}^{-1} & 0 \\
      -\apt ^{-*}\apot ^*\tilde V(t)_{21}V(t)_{11}^{-1} & 0
    \end{bmatrix}.
  \end{split}
  \end{equation}
  % \begin{equation}
  %   -\daE^*\dot \daP =  \daA^*\daP + \daP^*\daA - \daP^*\daB\daB^*\daP +
  %   \daC^*\daC, \quad \daE^*\daP=\daP^*\daE.
  % \end{equation}
\end{theorem}

\begin{proof}
  By Assumption \ref{ass:exist-are-sol-stab} and Theorem
  \ref{thm:dae-existence-dre-sol} the existence of $\daPD$ is ensured for $t\leq
  t_1$.  A direct computation that realizes the transformation of
  \eqref{eqn:dae-pptransf-pencil} for the associated Riccati equation reveals that \daPD~solves the generalized \emph{difference} Riccati equation (cp. \cite[p.
  994]{CalWW94}), i.e.
  \begin{equation*}
    -\daE^*\dot \daPD = \daAPs \daPD + \daPD^*\daAP - \daPD^*\daB\daB^*\daPD,
    \quad \daE^*\daPD(t_1)=\daF^*\daF - \daE^*\daPP.
  \end{equation*}
  % solves the generalized Riccati equation
  Since the second block column of \daPD~is zero -- as it follows directly from
  $\daPD = \daP-\daPP$ -- and $\apt$ is invertible, equation
  \eqref{eqn:dae-dif-ric} implies that the left-lower block of $\daPD$ is given
  as $ -\apt ^{-*}\apot ^*\daPDii(t)$ and that $\daPDii$ is the solution to the
  Riccati equation
  \begin{equation*}  % \label{eqn:dae-dapdi-ric}
    -\ddaPDii = \bAs \daPDii + \daPDii \bA - \daPDii \bB\bB^*\daPDii, \quad
    \daPDii(t_1) = S-\daPPii,
  \end{equation*}
  which in particular means that this Riccati equation has a global solution
  for $t\leq t_1$ despite the possibly indefinite initial condition. 
  
  From the existence of the Riccati solution \daPDii, we can infer the
  invertibility of $V_{11}(t)$ from the relation
  \begin{equation*}
    -\dot V_{11}(t) = (\bA-\bB\bB^*\daPDii(t))V_{11}(t), \quad V_{11}(t_1) = I,
  \end{equation*}
  which is a consequence of \emph{Radon's Lemma} (see, e.g, \cite[Thm. 4.1.1]{AboFIJ03}).

  Uniqueness follows from $V_{11}$ and $\tilde V_{21}$ being solutions to an ordinary
  linear differential equation, namely \eqref{eqn:dae-reduced-cl-hamiltonian}. 

  Thus, \daPD~as defined in \eqref{eqn:form-of-PD} is well-defined. 
  
  Following the lines of a proof for a standard result for ODEs \cite[180
  Theorem]{CalD91}, we directly
  confirm that it \daPD~solves the generalized difference Riccati equation
  \eqref{eqn:dae-dif-ric}. 
  
  For that we find that the summands of the left upper block of $-\daE ^* \dot
  \PD$ which are given as
\begin{equation*}
  -\frac{d}{dt} V_{21}V_{11}^{-1} =
  -\dot{\tilde V}_{21}V_{11}^{-1} + 
\tilde V_{21}V_{11}^{-1}\dot V_{11}V_{11}^{-1},
\end{equation*}
and which, by means of the equations for $\dot{ \tilde V}_{21}$ and $\dot V_{11}$ in
\eqref{eqn:dae-hami-aplus-resort} 
with $V_{12}$ and $V_{22}$ resolved via \eqref{eqn:resolve-v12-v22}, rewrite as
\begin{equation*}
  \begin{split}
    -\dot V_{21}V_{11}^{-1}&=(\apo ^* V_{21} - \apto ^*
    \aptms \apot ^* V_{21}) V_{11}^{-1}\\
                           &=(\apo ^*  - \apto ^* \aptms \apot ^* ) \daPDii\\
                           &=\apo ^* \daPDii + \apto ^* \aptms \apot ^* \daPDit
    \end{split}
  \end{equation*}
  and
  \begin{equation*}
    \begin{split}
      V_{21}&V_{11}^{-1}\dot V_{11}V_{11}^{-1}= \\
%      =&\quad V_{21} V_{11}^{-1} \left(\apo  V_{11} - \apot  \apt ^{-1} \left(\apto  V_{11} - \bto V_{21} + \btt \aptms \apot ^* V_{21}\right) - \boo V_{21} + \btos \aptms \apot ^* V_{21}\right) V_{11}^{-1} 
      &\quad =V_{21} V_{11}^{-1} \bigl([\apo  - \apot  \apt ^{-1} \apto]  V_{11} \\
      &\phantom{\quad =V_{21} V_{11}^{-1} \bigl(} -[B_1 -\apot \apt ^{-1}B_2
      ][B_1 -\apot \apt ^{-1}B_2 ]^* V_{21}\bigr) V_{11}^{-1} \\
      &\quad=\daPDii [\apo  - \apot  \apt ^{-1} \apto]  \\
      &\phantom{\quad=}- \daPDii [B_1 -\apot \apt ^{-1}B_2 ][B_1 -\apot \apt ^{-1}B_2 ]^* \daPDii 
      \\&\quad=\daPDii \apo  + \daPDit^*\apto  \\
        &\phantom{\quad=}- \daPDii B_1B_1^*\daPDii
      -\daPDii B_1B_2^*\daPDit-\daPDit^* B_2B_1^*\daPDii-\daPDit^*
      B_2B_2^*\daPDii,
      % \\=&\quad \daPDii \apo -\daPDii  \apot  \apt ^{-1} \left(\apto - \bto \daPDii + \btt \aptms \apot ^* \daPDii \right)- \daPDii b_{1} b_{1}^* \daPDii + \daPDii \btos \aptms \apot ^* \daPDii
    \end{split}
  \end{equation*}
sum up to the left upper block of 
\begin{equation*}
\APs \PD + \PD^*\AP - \PD^*\daB \daB^*\PD.
\end{equation*}
By the zero pattern of $\daPP$, the other blocks of $\PD^*\daB \daB^*\PD $ are
zero, whereas the other possibly nonzero blocks of $\daAPs\PD$ and $\PD^*\daAP$
sum up to zero, respectively, because of how the blocks of $\PD$ are related.
Thus, 
\begin{equation*}
-\daE ^*\dot \PD = \APs \PD + \PD^*\AP - \PD^*\daB \daB^*\PD
\end{equation*}
is fulfilled. 
Finally, by the structure of \daE, \daF, and, \daPD, the initial condition
$\daE^*\daPD(t_1)=\daF^*\daF - \daE^*\daPP$ reduces to $\daPDii(t_1)=S-\daPPii$
which is fulfilled as $\tilde V_{21}(t_1)V_{11}(t_1)=\tilde
V_{21}(t_1)=S-\daPPii$.
\end{proof}

From Theorem \ref{thm:v11-gDRE-PD}, we can directly deduce necessary conditions
for the convergence of $\daP(t)$ towards $\daPP$ as $t_1\to\infty$; cp. Lemma
\ref{lem:stilde} for the standard ODE case.

\begin{corollary}[of Theorem \ref{thm:v11-gDRE-PD}]\label{cor:dae-pdi-conv}
  The left upper block \daPDii~of \daPD~as defined in
  \eqref{eqn:form-of-PD} satisfies the relation
  \begin{equation}\label{eqn:dae-pdii-expformula}
    \daPDii(t) = \mxp{(t_1-t)\bA^*}\tbs(t_1-t)\mxp{(t_1-t)\bA}
  \end{equation}
  where
  \begin{equation}\label{eqn:dae-stilde}
  \tbs(\tau) := (S-\daPPii)[I+\bW(\tau)(S-\daPPii)]^{-1} 
\end{equation}
with
\begin{equation*}
  \bW(\tau): = \int_0^\tau \mxp{s\bA}\bB\bB^*\mxp{s\bA^*} \inva s
\end{equation*}
being well-defined. Moreover, $\daPDii(t) \to 0$ exponentially as $t_1\to \infty$, if, and only if, 
  \begin{equation*}
    I + \bW(S-\daPPii)
  \end{equation*}
  is nonsingular, where 
  \begin{equation*}  % \label{eqn:cl-findyn-reach-gram}
    \bW:=\lim_{\tau\to\infty}\bW(\tau).
  \end{equation*}
\end{corollary}

\begin{proof}
  Relation \eqref{eqn:dae-pdii-expformula} and \eqref{eqn:dae-stilde} follow
  from the \emph{variation of constants} formular applied to
  \eqref{eqn:dae-reduced-cl-hamiltonian} that gives
  \begin{equation*}
    \tilde V_{21}(t) = \mxp{-(t-t_1)\bA^*}(S-\daPPii)
  \end{equation*}
  and
  \begin{equation*}
  \tilde V_{11}(t) =
\mxp{(t-t_1)\bA}\bigl (I+\int_{t_1}^t\mxp{-(s-t_1)\bA}\bB\bB^*\mxp{-(s-t_1)\bA^*}\inva
s (S-\daPPii)\bigr )
  \end{equation*}
  and the invertibility of $V_{11}(t)$.
  By stability of $\bA$ (cp. Lemma \ref{lem:findyn-stab}) well posedness of the
  improper integral that defines $\bW$ is ensured, and the equivalence of
  $\daPDii(t)\to 0$ and invertibility of $I + \bW(S-\daPPii)$ follows by
  \cite[Lem. 3]{CalWW94}.
\end{proof}

The condition for the convergence motivates the following assumption:

\begin{assumption}\label{ass:dae-stilde-invrtbl}
  Consider Problem \ref{prb-dae-fiti-optcont}, let Assumption
  \ref{ass:exist-gare-sol-stab} hold and let $\daE$ be semi-explicit as in
  Assumption \ref{ass:E-semexp} and \daAP~and $\daB \daB^*$ be partitioned as in
  \eqref{eqn:partition-ABV} and let $\daF$ be compatible with $\daE$ as in
  Assumption \ref{ass:rangeF-rangeE}. The matrix
  \begin{equation*}
    I + \bW(S-\daPPii)
  \end{equation*}
  is nonsingular, where 
\begin{equation*}
  \bW: = \int_0^\infty \mxp{s\bA}\bB\bB^*\mxp{s\bA^*} \inva s
\end{equation*}
  is the closed loop \emph{finite dynamics} reachability Gramian and where $S$
  is the left upper block of $\daF^*\daF$; cp. \eqref{dae:fsf}.
\end{assumption}

Another outcome of the existence of this structured solution to the gDRE is that
the corresponding closed loop system does not generate impulses. Since, the
closed loop system is time varying, the definition of \emph{impulse-freeness}
(Def. \ref{def:dae-imp-free}) does not apply. Instead, we directly derive a
representation of the closed loop system in which the differential and algebraic
parts are decoupled. 
% the concept of \emph{strangeness freeness} \cite{KunM06} that is closely connected to and has the same implications as \emph{impulse freeness}.

\begin{corollary}[of Theorem \ref{thm:v11-gDRE-PD}]\label{cor:dae-cl-impfree}
  The closed loop system
  \begin{equation*}
    \daE \dot x(t) = (\daA -\daB \daB^* \daP(t))x(t), \quad \daE x(0) = \daE
    x_0,
  \end{equation*}
  can be written in decoupled form for $x=(x_1, x_2)$:
  \begin{equation}\label{eqn:cl-dcpld}
    \dot x_1(t) = \hat{\mathcal A}_1(t) x_1(t), \quad 0 = x_2(t) - \hat{\mathcal
    A}_2(t)x_1(t), \quad \mathcal E x(0)=\mathcal Ex_0.
  \end{equation}
  % i.e. the closed-loop system is \emph{strangeness free}.
\end{corollary}

\begin{proof}
  By the particular structure of \PD, it follows that the left lower block of
  $\daA -\daB \daB^*P(t) = \daA -\daB \daB^*(\PP + \PD)$ equals the left lower block of \AP, namely
  $\apt $. Since $\apt $ is invertible, with the help of corresponding Schur
  complement, the representation \eqref{eqn:cl-dcpld} of the DAE follows with 
  $\hat{\mathcal A}_1(t) := (\bA - \bB\bB^*\daPDii(t))$ and $\hat{\mathcal
  A}_2(t):= -\aptmo \apto + \aptmo B_2 \bB^* \daPDii(t)$.

  % defined by the
  %(time-dependent) coefficients $(\daE , \daA -\daB \daB^*\daP)$ is strangeness
  % free; cp. \cite[Thm. 3.17]{KunM06}.
\end{proof}

\begin{remark}
  From the representation \eqref{eqn:cl-dcpld} one can read off, that an
  inconsistent initial value will make in $x_2$ discontinuous at $t=0$ but will
  not induce impulses. In a more general context, the existence of such a
  decoupled representation is referred to as \emph{strangeness free}; cp.
  \cite[Thm. 3.17]{KunM06}.
\end{remark}
% \todo{conditions -- what is the relevant $C$ here}

Next, we show the turnpike property of the homogeneous
optimal control problem with DAE constraints. % with respect to the zero state and zero control action.

\begin{theorem}\label{thm:dae-turnpike}
  Under the assumptions of Theorem \ref{thm:v11-gDRE-PD} and under Assumption
  \ref{ass:dae-stilde-invrtbl}, % the optimal control problem 
  Problem
  \ref{prb-dae-fiti-optcont} with $\yc=0$ and $\ye=0$ has a solution $(x, u) $ which 
   fulfills the estimate
  \begin{equation*}
    \|x(t) \| \leq \const (\mxp{t\bar \sigma} + \mxp{(t_1 -t)\bar \sigma})
  \end{equation*}
  and
  \begin{equation*}
    \|u(t) \| \leq \const (\mxp{t\bar \sigma} + \mxp{(t_1 -t)\bar \sigma})
  \end{equation*}
  with $\const $ independent of $t_1$ and where $\bar \sigma < 0$ is the spectral abscissa of $\bA$.
\end{theorem}

\begin{proof}
  With $(V_{11}, V_{12}, \tilde V_{21}, \tilde V_{22})$ solving
  \eqref{eqn:dae-aplus-hami-sys}, with the relation \eqref{eqn:resolve-v12-v22},
  and with $V_{11}(t)$ being invertible, we have that 
  \begin{equation*}
    \begin{split}
    \begin{bmatrix}
      x_1 \\ x_2 \\ \tilde p_1 \\ \tilde p_2 
    \end{bmatrix}
    &= 
    \begin{bmatrix}
      V_{11} \\ V_{12} \\ \tilde V_{21} \\ \tilde V_{22} 
    \end{bmatrix} V_{11}^{-1}(t_0)x_0
     =
    \begin{bmatrix}
      I & 0 \\ s_{11} & s_{12} \\ 0 & I \\ 0 & s_{22} 
    \end{bmatrix}
    \begin{bmatrix}
      V_{11} \\ \tilde V_{21}
    \end{bmatrix} V_{11}^{-1}(t_0)x_0 \\
    & =
    \begin{bmatrix}
      I & 0 \\ \apt ^{-1} \apto  & -\apt ^{-1} \bto +\apt ^{-1}\btt \apt ^{-*}\apot ^* \\
      0 & I \\ 0 & \apt ^{-*}\apot ^* 
    \end{bmatrix}
    \begin{bmatrix}
      V_{11} \\ \tilde V_{21}
    \end{bmatrix} V_{11}^{-1}(t_0)x_0.
  \end{split}
  \end{equation*}
  defines the solution to \eqref{eqn:dae-aplus-hami-sys}, and, in particular, the optimal state as in
  \eqref{eqn:dae-hami-sys}. Multiplication of the first row gives that 
  \begin{equation*}
    V_{11}^{-1}(t)x(t) = V_{11}^{-1}(t_0)x(t_0)
  \end{equation*}
  so that, with $\daPDii = V_{21}V_{11}^{-1}$ (cp. Thm. \ref{thm:v11-gDRE-PD}), we
  get the following formula for the optimal state
  \begin{equation*}
    \begin{bmatrix}
      x_1(t) \\ x_2(t)
    \end{bmatrix}=
    \begin{bmatrix}
      V_{11}(t)V_{11}(t_0)^{-1}x_1(t_0) \\
      \apt ^{-1} \apto x_1(t) - 
      [\apt ^{-1} \bto +\apt ^{-1}\btt \apt ^{-*}\apot ^*]\daPDii(t)x_1(t)
    \end{bmatrix}.
  \end{equation*}
  Then the turnpike property for $x_1$ follows by the arguments for the ODE
  case \cite[Thm. 4]{CalWW94} as follows: 
  From
  \begin{equation*}
    \|x_1(t) - \mxp{\bA t} \| \leq \const  \mxp{t_1\bsigma }\mxp{(t_1-t)\bsigma};
  \end{equation*}
cp. \cite[Eqn. (63)]{CalWW94}, we conclude with $\mxp{t_1 \bsigma} \leq 1$
independent of $t_1$, $\bA$ being stable, and an application of the triangle
inequality as in \eqref{eqn:tp-from-calww} that 
  \begin{equation*}
    \|x_1(t) \| \leq \const  (\mxp{(t_1-t)\bsigma}+\mxp{t\bsigma}).
  \end{equation*}
  The same type of estimate for $x_2$ follows from the turnpike property of $x_1$ since
  $\daPDii$ is
  bounded by Corollary \ref{cor:dae-pdi-conv}.

  The turnpike estimate for $u(t) = -\daB\daB^*\daP(t)x(t) =
  -\daB\daB^*\daPP x(t)-\daB\daB^*\daPD(t)x(t)$ 
  follows as in \eqref{eqn:turnpike-u}.
  % \todo{$x_2(0)$??}
\end{proof}

We summarize and comment on the assumptions and the results of this chapter on
the optimal control of the linear descriptor system \eqref{eqn:dae} with a
quadratic cost functional defined in Problem \ref{prb-dae-fiti-optcont}.

\vspace{.075in}
% \noindent
\fbox{
  \begin{minipage}{.9\linewidth}
Assumptions:
\begin{enumerate}
  \item Without loss of generality: $\daE$ is semi-explicit (Assumption
    \ref{ass:E-semexp}).
  \item To enable existence of solutions to the first order optimality
    conditions \eqref{eqn:first-order-opti-dae}: Compatibility of $\daF$ and
    \daE~(Assumption \ref{ass:rangeF-rangeE}).
  \item In line with the basic assumption for the ODE case: Existence of a
    stabilizing solution to the generalized algebraic Riccati equation
    \eqref{eqn:dae-alg-ric} including regularity of the matrix pair (Assumption
    \ref{ass:exist-gare-sol-stab}).
  \item To ensure existence of global solutions to the reduced closed loop
    Riccati equation \eqref{eqn:dae-dre-poo-resolved}: A spectral condition on
    the coefficients (Assumption \ref{ass:dae-zdg-cond}).
  \item In line with the relevant condition in the ODE case: Compatibility of
    the terminal constraint $S$, the solution to the gARE
    \eqref{eqn:dae-alg-ric}, and the relevant \emph{reachability Gramian}
    (Assumption \ref{ass:dae-stilde-invrtbl}).
\end{enumerate}
\end{minipage}
}

\vspace{.075in}

Certainly, Assumption \ref{ass:dae-zdg-cond} is somewhat unpleasant and, because
of its dependency on $\daPt$ not readily confirmed or discarded for a given
system. Nonetheless, it generalizes the standard assumption for ODE systems that
ensures the definiteness of the cost functional in the presence of cross terms
in the costs or a feedthrough term in the system.

With these assumptions, the following results have been derived:

\vspace{.075in}
% \noindent
\fbox{
  \begin{minipage}{.9\linewidth}
Summary of results:
\begin{enumerate}
  \item Existence of solutions to the generalized differential Riccati equation
    (Theorem \ref{thm:dae-existence-dre-sol}).
  \item Representation of the difference $\daP(t)-\daPP$ that implies
    that the closed loop system is \emph{impulse-free} (Theorem
    \ref{thm:v11-gDRE-PD} and Corollary \ref{cor:dae-cl-impfree}).
  \item Convergence of $\daP(t) \to \daPP$ as $t_1\to\infty$ and turnpike property
    of the \emph{homogeneous} optimal control problem with DAE constraints
    (Corollary \ref{cor:dae-pdi-conv} and Theorem \ref{thm:dae-turnpike}).
\end{enumerate}
\end{minipage}
}

\section{The Affine DAE LQR Problem}\label{sec:dae-affine-lqr}
In this section, we study the optimal control problem with nonzero target states
$\yc$ and $\ye$ and get back to the question of what \emph{the} steady-state
optimal control problem is for a descriptor system.

Similarly to the ODE case, the feedthrough $w$ is defined via
\begin{equation*}
  -\daE\dot w = (\daA^*-\daPs(t)\daB \daB^*)w - \daC^*\yc, \quad \daE^*w(t_1) =
  -\daF^*\ye,
\end{equation*}
which we rewrite as 
\begin{equation*}
  -\daE\dot w = \daAPs w - \daPDs(t)\daB \daB^* w - \daC^*\yc, \quad \daE^*w(t_1) =
  -\daF^*\ye.
\end{equation*}

We partition the variables and coefficients
\begin{equation*}
  w = 
  \begin{bmatrix}
    w_1 \\ w_2
  \end{bmatrix}, \quad
  \daC^* = 
  \begin{bmatrix}
    C_1^* \\ C_2^*
  \end{bmatrix}, \quad \text{and }
  \daF^* = 
  \begin{bmatrix}
    F_1^* \\ 0
  \end{bmatrix}
\end{equation*}
in accordance with $\daE$, $\daAP$, $\daB \daB^*$, and $\daPD$, as in
\eqref{eqn:semexp-E}, \eqref{eqn:partition-ABV}, and \eqref{eqn:form-of-PD},
respectively, and write
\begin{equation*}
  \begin{split}
    \daAPs - \daPDs(t) \daB \daB^* &=
  \begin{bmatrix} \apo^* & \apto^* \\ \apot^* & \apt^* \end{bmatrix}
  -
    \begin{bmatrix}
      \daPDii(t) & -\daPDii(t)\apot\aptmo \\0& 0
    \end{bmatrix}
  \begin{bmatrix}
    \boo & \bto \\ \btos & \btt
  \end{bmatrix}\\
  % &=\begin{bmatrix} 
  %   \apo^*-\daPDii(t)[\boo-\apot\aptmo\bto] & \apto^{-*}
  %   -\daPDii(t)[\btos-\apot\aptmo\btt]\\
  %   \apot^* & \apt^* 
  % \end{bmatrix} \\
  &=\begin{bmatrix} \apo^*-\daPDii(t)\bB B_1^* & \apto^{-*}
  -\daPDii(t)\bB B_2^* \\
\apot^* & \apt^* \end{bmatrix}.
\end{split}
\end{equation*}
In what follows, we will omit the time dependency of \daPD. With the relation
\begin{equation}\label{eqn:dae-wtwo}
  w_2= -\aptms\apto^*w_1 +\aptms C_2^*\yc
\end{equation}
we can eliminate $w_2$ from the equations and consider only
% \begin{equation}
%   \begin{split}
%   -\dot w_1 = \bigl[\apo^*-\apto^*\aptms\apot^*]-\daPDii[\boo-\apot\aptmo\bto] +
%   \daPDii[\btos-\apot\aptmo\btt]\aptms\apto^*\bigr]w_1\\
%   + C_1^*\yc - \bigl[ \apto^{-*} -\daPDii(t)[\btos-\apot\aptmo\btt]\bigr]\aptms C_2^*\yc
% \end{split}
% \end{equation}
% which we write as
\begin{equation}\label{eqn:opti-wone}
    -\dot w_1 = (\bAs-\daPDii(t)\bB\bB^* )w_1 - \bC ^*\yc + \daPDii(t)\bB B_2^*\aptms
    C_2^*\yc,
\end{equation}
with the abbreviations
\begin{equation*}
  \bA:= \apo-\apot\aptmo\apto, \quad \bB:=B_1-\apot\aptmo B_2, \quad\text{and }
  \bC:=C_1-C_2\aptmo\apto
\end{equation*}
as they have been used before.

By the same procedure, we derive the expressions for the parts of the optimal
state as
\begin{equation}\label{eqn:opti-xtwo}
  x_2 = -\aptmo \apto x_1 + \aptmo B_2 \bB^* (\daPDii x_1 + w_1) +
  \aptmo\btt\aptms C_2 \yc
\end{equation}
and
\begin{equation}\label{eqn:opti-xone}
  \dot x_1 = (\bA - \bB\bB^*\daPDii(t))x_1- \bB\bB^*w_1 - \bB B_2 ^* \aptms C_2 \yc.
\end{equation}

% We first state the turnpike property for the part $x_1$ and later conclude the estimates for the whole state $x$. 

The preceding derivations show that if the Riccati solution exists, then the
optimal states $x=(x_1, x_2)$ decouple such that $x_1$ reads like a solution to an optimal
control problem with ODE constraints, and $x_2$ is in a direct algebraic
relation with $x_1$. Accordingly, we can state the turnpike property for Problem
\ref{prb-dae-fiti-optcont} with similar arguments as for the standard LQR case.

\begin{theorem}\label{thm:dae-affine-turnpike}
  Consider Problem \ref{prb-dae-fiti-optcont} with
  the costs defined through $t_1$, $\daC$, $\daF$, and target states $\yc$ and
  $\ye$, and subject to the DAE \eqref{eqn:dae} with coefficients $(\daE, \daA,
  \daB)$.

  Assume that $\daE$ is semi-explicit, that $(\daE, \daA)$ is regular, and that
  the gARE \eqref{eqn:dae-alg-ric} has a stabilizing solution $\daPP$
  (Assumptions \ref{ass:exist-gare-sol-stab} and \ref{ass:E-semexp}).
  
  Assume that $\daF$ is compatible with $\daE$ so that with Assumptions
  \ref{ass:dae-zdg-cond} the gDRE \eqref{eqn:dae-dif-ric} has a solution $\daP$
  with $\daPD=\daP-\daPP$ as in \eqref{eqn:form-of-PD}.  

  Assume that the relevant part of $\daF$ is compatible with the relevant part
  of $\daPP$ such that Assumption \ref{ass:dae-stilde-invrtbl} is fulfilled.
  
  Then the optimal control has a solution $(x,u)$ with $x$ and $u$ satisfying
  the estimates
  \begin{equation*}
    \|x(t)-x_s\| \leq \const (\mxp{t\bsigma}+\mxp{(t_1-t)\bsigma},
  \end{equation*}
  and
  \begin{equation*}
    \|u(t)-u_s\| \leq \const (\mxp{t\bsigma}+\mxp{(t_1-t)\bsigma},
  \end{equation*}
  with $\const $ independent of $t_1$ and where $\bsigma < 0$ is the spectral abscissa
  of $\bA$ and 
  \begin{equation}\label{eqn:dae-x-turnpike}
    x_s = 
    % \begin{bmatrix}
    %   (\bA^{-1}\bB\bB^*\bA^{-*}\bC  + \bA^{-1}\bB B_2^* \aptms C_2 )\yc
    % \\
    % -\aptmo \apto(\bA^{-1}\bB\bB^*\bA^{-*}\bC  + \bA^{-1}\bB B_2^* \aptms
    % C_2)\yc + \aptmo B_2 (\bB^*\bA^{-*}\bC^* + B_2^* \aptms C_2^*)\yc 
    % \end{bmatrix}
    \begin{bmatrix}
      x_{s;1}
    \\
    -\aptmo \apto x_{s;1} + \aptmo B_2 (\bB^*\bA^{-*}\bC^* + B_2^* \aptms C_2^*)\yc 
    \end{bmatrix}
  \end{equation}
  with $x_{s;1}:=
      (\bA^{-1}\bB\bB^*\bA^{-*}\bC^*  + \bA^{-1}\bB B_2^* \aptms C_2^* )\yc$
  and
  \begin{equation}\label{eqn:dae-turnpike-u}
    u_s = - \daB^*\daPP x_s - \bB^*\bA^{-*}\bC^*\yc - B_2^*\aptms C_2^*\yc.
  \end{equation}
\end{theorem}

\begin{proof}
  We consider equations \eqref{eqn:opti-wone} and \eqref{eqn:opti-xone} for the
  parts $w_1$ and $x_1$, respectively. As laid out in the proof of Theorem
  \ref{thm:v11-gDRE-PD}, the part $\daPD$ solves the
  differential Riccati equation
  \begin{equation*}
    -\ddaPDii = \bAs \daPDii + \daPDii \bA - \daPDii \bB\bB^*\daPDii, \quad
    \daPDii(t_1) = S_1-\daPPii
  \end{equation*}
  and exists for all $t\leq t_1$. Thus, Lemma \ref{lem:fundamental-solution} applies and provides the formulas for the relevant fundamental solution as
  \begin{equation*}
    U(t) =
    \mxp{-(t_1-t)\bA}\bigl(I-[\bW-\mxp{(t_1-t)\bA}\bW\mxp{(t_1-t)\bAs}](\daPPii-S_1)\bigr).
  \end{equation*}
  As for the ODE case, we conclude that the feedthrough can be written as 
  \begin{equation}\label{eqn:dae-wones}
    w_1(t) = \bA^{-*}\bC^*y_c + g(t, t_1)
  \end{equation}
  with a remainder term $g(t,t_1)$ that is dominated by $\mxp{(t_1-t)\bA}$; cp.
  \eqref{eqn:w-homo} and \eqref{eqn:w-partic}. The only difference to the ODE case
  lies in the additional term
  $\daPDii(t)\bB\aptms C_2^*\yc$ to $w_1$ as defined \eqref{eqn:dae-wones} which,
  however, can be included in the estimates by the decaying behavior of
  $\daPDii$ as it is ensured by Assumption \ref{ass:dae-stilde-invrtbl} and
  Corollary \ref{cor:dae-pdi-conv}. And again, that part of $x_1$
  that cannot be bounded by $\mxp{(t_1-t)\bar \sigma}$, where $\bar \sigma$ is
  the spectral abscissa of $\bA$, is given by the integral operator $I_1$ (cp. \eqref{eqn:formula-x})
  applied to the constant parts in the right-hand side of \eqref{eqn:opti-xone}. Thus,
  the turnpike for $x_1$ is given by the constant part of
  \begin{equation*}
  \begin{split}
    -\int_0^t \mxp{(t-s)\bA}\bigl(\bB\bB^*&\bAms \bC^*y_c + \bB B_2^* \aptms C_2 \yc\bigr)\inva s \\
     &= (I-\mxp{t\bA})\bA^{-1} \bigl [\bB\bB^*\bAms \bC^*y_c + \bB B_2^*\aptms C_2 \yc\bigr].
  \end{split}
\end{equation*}
which is as in the first component of \eqref{eqn:dae-x-turnpike}. The turnpike
for $x_2$ as in the second component of \eqref{eqn:dae-x-turnpike} follows from
formula \eqref{eqn:opti-xtwo} in combination with the decaying behavior of
$\daPDii$ and the estimate for $w_1$ given in \eqref{eqn:dae-wones}.

With the formulas \eqref{eqn:dae-wones} and \eqref{eqn:dae-wtwo}
  for $w_1$ and $w_2$, the optimal control writes as
\begin{align*}
  u(t) &= - \daB^*(\daP x(t) + w(t)) \\
       &= -\daB^* \daPP x(t) -\daB^* \daPD(t)x(t)
  -B_1 ^*w_1(t) - B_2^* w_2(t) \\
       &= -\daB^*\daPP (x(t)-x_s) -\bB^*\daPDii(t)x_1(t)-\bB^* g(t, t_1) \\
       &\phantom{=\quad}-
       \bB^*\bA^{-*}\bC^*y_c-B_2^*\aptms C_2^*\yc - \daB^*\daPP x_s \\
       &= -\daB^*\daPP (x(t)-x_s) -\bB^*\daPDii(t)x_1(t)-\bB^* g(t, t_1) + u_s
\end{align*}
from where the estimate \eqref{eqn:dae-turnpike-u} for $u(t)-u_s$ follows with
the same arguments as for \eqref{eqn:turnpike-u}.
\end{proof}

With that we have established the existence of the turnpike for a class of
linear-quadratic optimization problem with DAE constraints. 

It
remains to investigate whether, like in the ODE case, the turnpike can be
defined as the solution to an associated steady-state optimal control problem. For that we observe that with $p_s:=\daPP x_s+w_s$, with 
$w_s$ being the time constant parts of $w$ and with $u_s$ as defined in
Theorem \ref{thm:dae-affine-turnpike}, the triple $(x_s, p_s, u_s)$ is a
critical point 
of the \emph{Lagrange function} 
\begin{equation*}
  \mathcal L( x,  p, u) = \frac 12 \|\daC  x - y_c \|^2 + \frac
  12 \| u\|^2 +
   p^*(\daA  x+\daB  u).
\end{equation*}
Still, since the omission of the time derivative discards the DAE structure, the turnpike is not that solution that arises from 
\begin{equation*}
  \frac 12 \|\daC x - y_c \|^2 + \frac 12 \|u\|^2 \to \min_u, \text{ subject to
  } \daA x + \daB u =0. 
\end{equation*}
In particular, the feedback structure $u_s=-\daB\daB^*\daPP x_s$ that makes the
left-lower block of $\daA-\daB\daB^*\daPP$ regular is not ensured.

\section{Conclusion and Discussion}

The presented results show that classical system theoretic results well apply to
prove turnpike properties for LQR problems constraint by standard linear state
space systems. Under the assumption of impulse controllability, a descriptor
system can be controlled such that it is basically an ODE with an additional but
well separated algebraic part so that similar arguments can be used to confer
turnpike properties of LQR problems with DAE constraints.

The characterization of the turnpike for DAE problems as an optimal steady
state is still undecided. For time-dependent problem, we succeeded in
establishing an \emph{underlying Hamiltonian ODE system} for the optimality conditions 
(cp. \eqref{eqn:dae-reduced-cl-hamiltonian}). In the steady-state regime,
similar operations did not lead to a system that allowed for a characterization
as an optimality system like \eqref{eqn:fon}. A possible approach would be to
establish results for impulse-free matrix pairs $(\mathcal E, \mathcal A)$, e.g.
through the equivalence to ODE systems with feedthrough, and then extend to
impulse-controllable systems, e.g., through the \emph{feedback equivalence
form} (see \cite{ReiRV15}). 

In line with the literature on turnpike phenomena in control systems, natural
extensions of the presented results could consider periodic orbits as turnpikes
(as in \cite{ArtL85, TreZZ18}),
PDE formulations (as in, e.g., \cite{GruSS19}), or particular nonlinear
phenomena (as in \cite{Pig20a,SakPZ19}) for the differential algebraic case.

As for the theory of control of DAEs, an immediate strengthening of the results
could be achieved by removing the assumption on impulse controllability. A more
general framework will also consider indefinite cost functionals and suitable
replacements for the Riccati equations, as they are used recent works on
singular feedback control \cite{BhaP19} or infinite time horizon problems
\cite{ReiV19}. The adaption of the concepts of \emph{Lur'e
equations} -- as they are a key tool in \cite{ReiV19} -- to tracking problems on
finite time horizons has not been investigated yet. However, the related concept
of \emph{storage functions} and their connection to turnpike properties, as
discussed in \cite{GrG21}, may well be used for a general theory for DAEs that
does not resort to Riccati equations.

% Another issue, in particular in view of numerical realizations, is the
% nonuniqueness of $\daPt$ that, from an optimistic point of view, could be
% exploited for the design of optimal feedback laws; cp. also \cite{KatM92}. A
% general open research issue is the development of numerical schemes for the
% solution of the generalized algebraic and differential Riccati equations. 

\section*{Acknowledgments}
This project has received funding from the European Research Council (ERC) under the European Union's Horizon 2020 research and innovation programme (grant agreement No. 694126-DyCon). 

The work of the second author has been funded by the Alexander von
Humboldt-Professorship program, the European Union's Horizon 2020 research and
innovation programme under the Marie Sklodowska-Curie grant agreement
No.765579-ConFlex, grant MTM2017-92996-C2-1-R COSNET of MINECO (Spain), ICON of
the French ANR and Nonlocal PDEs: Analysis, Control and Beyond, AFOSR Grant
FA9550-18-1-0242, and Transregio 154 Project 
\emph{Mathematical Modelling, Simulation and Optimization using the Example of
Gas Networks} of the German DFG.
% \bibliographystyle{siamplain}
% \bibliography{bib_jh,tpric}
% \bibliography{tp-ric-dae}

\end{document}